\documentclass[aop]{imsart}

\RequirePackage[OT1]{fontenc}
\RequirePackage{amsthm,amsmath,amsfonts,amssymb}
\RequirePackage[authoryear]{natbib}
\RequirePackage[colorlinks,citecolor=blue,urlcolor=blue]{hyperref}
\RequirePackage{graphicx}

\usepackage{lmodern} 
\usepackage{mathtools,enumerate,scalerel}
\usepackage[format=plain,labelfont=it,textfont=it]{caption}
\usepackage{subfig}
\startlocaldefs

\newcommand{\e}{\varepsilon}
\newcommand{\ii}{\mathrm{i}}
\newcommand{\N}{\mathbb N}
\newcommand{\R}{\mathbb R}
\newcommand{\C}{\mathbb C}

\newcommand{\bb}[1]{\boldsymbol{#1}}
\newcommand{\rd}{{\rm d}}
\newcommand{\m}{\mathrm{Leb}}
\renewcommand{\Re}{\mathrm{Re}\hspace{0.9mm}}
\renewcommand{\Im}{\mathrm{Im}\hspace{0.9mm}}
\renewcommand{\P}{\mathbb{P}}

\newcommand{\E}{\mathbb{E}}
\newcommand{\OO}{\mathcal O}
\newcommand{\oo}{\mathrm{o}}
\newcommand{\leqdef}{\vcentcolon=}

\newtheorem{theorem}{Theorem}[section]
\newtheorem{lemma}[theorem]{Lemma}
\newtheorem{corollary}[theorem]{Corollary}
\newtheorem{proposition}[theorem]{Proposition}
\newtheorem{conjecture}[theorem]{Conjecture}
\newtheorem*{remark}{Remark}
\newtheorem*{notation}{Notation}
\numberwithin{equation}{section}

\endlocaldefs

\begin{document}

\begin{frontmatter}

    \title{Moments of the Riemann zeta function on short intervals of the critical line}
    \runtitle{Moments of the zeta function on short intervals}

    \begin{aug}
        \author[A]{\fnms{Louis-Pierre} \snm{Arguin}\thanksref{t1}\ead[label=e1]{louis-pierre.arguin@baruch.cuny.edu}},
        \author[B]{\fnms{Fr\'ed\'eric} \snm{Ouimet}\thanksref{t2}\ead[label=e2,mark]{ouimetfr@caltech.edu}}\\
        \and
        \author[B]{\fnms{Maksym} \snm{Radziwi\l\l}\thanksref{t3}\ead[label=e3,mark]{maksym@caltech.edu}}%
        \thankstext{t1}{L.-P.\ A.\ is supported in part by NSF Grant DMS-1513441 and by NSF CAREER DMS-1653602.}%
        \thankstext{t2}{F.\ O.\ is supported by postdoctoral fellowship from the NSERC (PDF) and the FRQNT (B3X).}%
        \thankstext{t3}{M.\ R.\ acknowledges support of a Sloan fellowship and NSF grant DMS-1902063.}%
        \address[A]{Baruch College and Graduate Center (CUNY), \printead{e1}%
        }\vspace{-2mm}%
        \address[B]{California Institute of Technology, \printead{e2,e3}%
        }%
        \runauthor{L.-P.\ Arguin, F.\ Ouimet and M.\ Radziwi\l\l}%
    \end{aug}

    \begin{abstract}
        We show that as $T\to \infty$, for all $t\in [T,2T]$ outside of a set of measure $\oo(T)$,
        \begin{equation*}
            \int_{-\log^\theta T}^{\log^\theta T} |\zeta(\tfrac 12 + \ii t + \ii h)|^{\beta} \rd h = (\log T)^{f_{\theta}(\beta) + \oo(1)},
        \end{equation*}
        for some explicit exponent $f_{\theta}(\beta)$, where $\theta > -1$ and $\beta > 0$.
        This proves an extended version of a conjecture of \cite{MR3151088}.
        In particular, it shows that, for all $\theta > -1$, the moments exhibit a phase transition at a critical exponent $\beta_c(\theta)$,
        below which $f_\theta(\beta)$  is quadratic and above which $f_\theta(\beta)$  is linear.
        The form of the exponent $f_\theta$ also differs between mesoscopic intervals ($-1<\theta<0$) and macroscopic intervals ($\theta>0$), a phenomenon that stems from an approximate tree structure for the correlations of zeta.
        We also prove that, for all $t\in [T,2T]$ outside a set of measure $\oo(T)$,
        \begin{equation*}
            \max_{|h| \leq \log^\theta T} |\zeta(\tfrac{1}{2} + \ii t + \ii h)| = (\log T)^{m(\theta) + \oo(1)},
        \end{equation*}
        for some explicit $m(\theta)$.
        This generalizes earlier results of \cite{najnudel2018} and \cite{ABBRS_2019} for $\theta = 0$.
        The proofs are unconditional, except for the upper bounds when $\theta > 3$, where the Riemann hypothesis is assumed.
    \end{abstract}



\end{frontmatter}

\tableofcontents

\section{Introduction}

\subsection{Maxima and moments over large intervals}

    Understanding the growth of the Riemann zeta function $\zeta(s)$ on the critical line $\Re s = \tfrac 12$ is a central problem in number theory due, among other things, to its relationship with the distribution of the zeros of $\zeta(s)$, see e.g.\ Theorem 9.3 in \cite{MR882550}, and the more general subconvexity problem, see e.g.\ \cite{MR2653249,MR2680486}, and see \cite{MR1826269} for a general discussion.

    The Lindel\"of hypothesis predicts that, for any $\e > 0$ and all $t \in \mathbb{R}$, we have
    $|\zeta(\tfrac 12 + \ii t)| =\OO((1 + |t|)^{\e})$,
    whereas it follows from the Riemann hypothesis that
    \begin{equation}
        |\zeta(\tfrac 12 + \ii t)| = \OO\left(\exp \bigg( \Big ( \frac{\log 2}{2} + \oo(1) \Big ) \frac{\log t}{\log\log t} \bigg)\right), \quad \text{as } t\to\infty,
    \end{equation}
    see \cite{Chandee}.

    Unfortunately, there is a large gap between these conditional results and the best unconditional upper bounds, such as \cite{bourgain}, which shows that $|\zeta(\tfrac 12 + \ii t)| = \OO((1 + |t|)^{13/84 + \e})$ for any given $\e > 0$ and all $t \in \mathbb{R}$. Currently, the best unconditional lower bound,
    \begin{equation}
        \max_{t \in [0, T]} |\zeta(\tfrac 12 + \ii t)| \geq \exp \bigg( (\sqrt{2} + \oo(1)) \sqrt{\frac{\log T \log\log\log T}{\log\log T}} \bigg), \quad \text{as } T\to\infty,
    \end{equation}
    is established in \cite{tenenbaum} building on a method from \cite{Bondarenko}.

    The true order of the maximum of $|\zeta(\tfrac 12 + \ii t)|$ remains elusive to this day.
    A conjecture that we find plausible is stated in \cite{FGH07}, where it is conjectured based on probabilistic models that
    \begin{equation}
        \max_{t\in [0,T]} |\zeta(\tfrac{1}{2} + \ii t)| = \exp\left( \Big ( \frac{1}{\sqrt{2}}+\oo(1) \Big ) \sqrt{\log T\cdot \log\log T}\right), \quad \text{as } T\to\infty.
    \end{equation}

    Another set of central objects in the theory of the Riemann zeta function are the \textit{moments}
    \begin{equation}
        \frac{1}{T} \int_{T}^{2T} |\zeta(\tfrac 12 + \ii t)|^{\beta} \rd t, \quad \beta > 0.
    \end{equation}
    Their importance comes from their relationship to the size and zero-distribution of $\zeta(s)$. However, unlike the problem of understanding the size of the global maximum of $|\zeta(\tfrac 12 + \ii t)|$, we are in possession of widely believed conjectures regarding the behavior of moments. Following the work \cite{Keating-Snaith_2000}, it is expected that, for all $\beta > 0$,
    \begin{equation}\label{eqn: moment conjecture}
        \frac{1}{T}\int_T^{2T}|\zeta(\tfrac 12 + \ii t)|^\beta\rd t\sim C_\beta (\log T)^{\beta^2/4}, \quad \text{as } T\to\infty,
    \end{equation}
    and that the constant $C_\beta > 0$ factors into a product of two constants: one is computed from the moments of the characteristic polynomial of random unitary matrices, and the other is an arithmetic factor coming from the small primes.

    There are a few results supporting \eqref{eqn: moment conjecture}. First, the conjecture \eqref{eqn: moment conjecture} is known for $\beta = 2$ and $\beta = 4$ following the classical work of Hardy-Littlewood and Ingham.
    Upper bounds of the correct order of magnitude are established in \cite{heap-radziwill-soundararajan} for $0 < \beta \leq 4$.
    Meanwhile, lower bounds of the correct order of magnitude have been established for all $\beta \geq 2$ in \cite{RS}.
    Conditionally on the Riemann hypothesis, the correct order of magnitude of \eqref{eqn: moment conjecture} is known for all $\beta > 0$ (see \cite{Sound09,harper_sharp_13} for the upper bounds and \cite{HB1} for the lower bounds).

\subsection{Maxima and moments over short intervals}

    Motivated by the problem of understanding the global maximum, \cite{FyodorovHiaryKeating2012, MR3151088} initiated the question of understanding the true size of \textit{the local maximum of $\zeta(\tfrac{1}{2} + \ii t)$} by establishing a connection with log-correlated processes.
    If $\tau$ is sampled uniformly on $[T,2T]$, they conjectured that for any $0<\delta<1$, there exists $C=C(\delta)>0$ large enough and independent of $T$, such that with probability $1-\delta$,
    \begin{equation}\label{eqn: FHK}
        \max_{h\in[-1,1]}  \log |\zeta(\tfrac 12 + \ii \tau + \ii h)| - \big(\log\log T - \frac{3}{4} \log\log\log T\big) \in [-C,C].
    \end{equation}
    They also conjectured weak convergence, with a limiting tail of the form $C y e^{-2y}$. The leading order $\log \log T$ was proved in \cite{najnudel2018} (conditionally on the Riemann hypothesis for the lower bound) and in \cite{ABBRS_2019} unconditionally.
    The sharp upper bound was recently established in \cite{ABR-2020}.

    It is also conjectured in \cite{FyodorovHiaryKeating2012, MR3151088} (see Equations (14) and (2.30), respectively) that the moments in a short interval undergo a {\it freezing phase transition}, that is, the event
    \begin{equation}\label{eqn: FK}
        \int_{[-1,1]} |\zeta(\tfrac 12 + \ii \tau + \ii h)|^{\beta} \rd h =
        \begin{cases}
            (\log T)^{\beta^2 / 4 + \oo(1)}, & \text{ if } \beta \leq 2, \\[1mm]
            (\log T)^{\beta - 1  + \oo(1)}, & \text{ if } \beta > 2,
        \end{cases}
    \end{equation}
    has probability $1 - \oo(1)$ as $T\to \infty$.
    \cite{MR3151088} also state corresponding conjectures for mesoscopic intervals of length $\log^\theta T$ when $\theta \in (-1, 0)$, as well as finer asymptotics for the moments.

    In view of Equations \eqref{eqn: moment conjecture} and \eqref{eqn: FK}, an obvious question is to determine up to which  interval size the freezing phase transition persists.
    In this paper, we establish that freezing transitions occur exactly for interval sizes of order $\log^\theta T$ with $\theta > -1$.
    We also obtain the corresponding results for local maxima over such intervals.
    The following functions will be crucial to our analysis:
    \begin{equation}\label{eq:def:free.energy}
        \begin{aligned}
            &\text{$\theta\leq 0$:} \quad
            &&m(\theta) \leqdef 1 + \theta, \quad
            &&f_{\theta}(\beta) \leqdef
            \begin{cases}
                \frac{\beta^2}{4} (1 + \theta) + \theta, &\mbox{if } \beta \leq \beta_c (\theta)=2, \\[1mm]
                \beta m(\theta) - 1, &\mbox{if } \beta > \beta_c(\theta),
            \end{cases}
            \\[2mm]
            &\text{$\theta>0$:} \quad
            &&m(\theta) \leqdef \sqrt{1 + \theta}, \quad
            &&f_{\theta}(\beta) \leqdef
            \begin{cases}
                \frac{\beta^2 }{4} + \theta, &\mbox{if } \beta \leq \beta_c (\theta)=2 \sqrt{1 + \theta}, \\[1mm]
                \beta m(\theta) - 1, &\mbox{if } \beta > \beta_c(\theta).
            \end{cases}
        \end{aligned}
    \end{equation}

    \begin{theorem}[Moments]\label{thm: freezing}
        Let $\theta > -1$, $\beta > 0$ and $\e > 0$ be given.
        Let $\tau$ be a random variable uniformly distributed on $[T,2T]$.
        Then, as $T\to \infty$, we have
        \begin{equation}
            \P\Big(\int_{-\log^\theta T}^{\log^\theta T} |\zeta(\tfrac 12 + \ii \tau + \ii h)|^{\beta} \rd h < (\log T)^{f_{\theta}(\beta) - \e}\Big) = \oo(1).
        \end{equation}
        Moreover, if $\theta \leq 3$ or if the Riemann hypothesis holds, then as $T \rightarrow \infty$,
        \begin{equation} \label{eq:bbbb}
            \P\Big(\int_{-\log^\theta T}^{\log^\theta T} |\zeta(\tfrac 12 + \ii \tau + \ii h)|^{\beta} \rd h > (\log T)^{f_{\theta}(\beta) + \e}\Big) = \oo(1).
        \end{equation}
    \end{theorem}

    \begin{proof}
        For the upper bound, see Section~\ref{se:upperbounds}, and for the lower bound, see Proposition~\ref{prop:lower.bound.moments.zeta}.
    \end{proof}

    When $\beta > \beta_{c}(\theta)$, the moments exhibit \textit{freezing}, i.e.\ they are dominated by a few large values at the level of the local maximum of $|\zeta(\tfrac 12 + \ii \tau + \ii h)|$, $|h| \leq \log^\theta T$.
    Theorem~\ref{thm: freezing} also suggests that freezing does not occur for intervals larger than any fixed power of $\log T$, since $\beta_c(\theta)\to \infty$ as $\theta\to\infty$.
    We note that recently a sharp upper bound in the case $(\theta = 0, \beta = 2)$ has been established in \cite{Harper101}, thus refining the $(\log T)^{\e}$ factor appearing in \eqref{eq:bbbb} when $\theta = 0$ and $\beta = 2$.

    \begin{theorem}[Local maximum]\label{thm: max}
        Let $\theta > - 1$ and $\e > 0$ be given.
        Let $\tau$ be a random variable uniformly distributed on $[T,2T]$.
        Then, as $T\to \infty$, we have
        \begin{equation}
            \P\Big(\max_{|h| \leq \log^\theta T} |\zeta(\tfrac 12 + \ii \tau + \ii h)| < (\log T)^{m(\theta) - \e}\Big) = \oo(1).
        \end{equation}
        Moreover, if $\theta \leq 3$ or if the Riemann hypothesis holds, then as $T \rightarrow \infty$,
        \begin{equation}
            \P\Big(\max_{|h| \leq \log^\theta T} |\zeta(\tfrac 12 + \ii \tau + \ii h)| > (\log T)^{m(\theta) + \e}\Big) = \oo(1).
        \end{equation}
    \end{theorem}

    \begin{proof}
        For the upper bound, see Section~\ref{se:upperbounds}, and for the lower bound, see Proposition~\ref{prop:lower.bound.maximum.zeta}.
    \end{proof}

    It is instructive to put these results in the context of two well-known facts on $\zeta$.
    First, Selberg's central limit theorem, see for example \cite{selberg-1946,selberg-1989} or the simple proof in \cite{RadSou15}, states that, for any given $a < b$,
    \begin{equation}
        \P\Bigg(\frac{\log |\zeta(\tfrac 12 + \ii \tau)|}{\sqrt{\frac{1}{2} \log\log T}} \in (a, b)\Bigg) \xrightarrow{T\to\infty} \int_{a}^{b} \frac{e^{-u^2 / 2}}{\sqrt{2\pi}} \, \rd u.
    \end{equation}
    In other words, a typical value of $\log |\zeta(\tfrac 12 + \ii \tau)|$ is a Gaussian random variable of variance $\frac{1}{2}\log\log T$.
    This is consistent with the moment conjecture \eqref{eqn: moment conjecture} which gives a precise expression for the Laplace transform of $\log |\zeta(\tfrac 12 + \ii \tau)|$.
    Second, since $\zeta(\tfrac 12 + \ii t)$ varies on the scale of $(\log T)^{-1}$ for $T \leq t \leq 2T$, the analysis of large values should be reducible to a discrete set of $(\log T)^{1+\theta}$ points.
    Putting these two facts together, one expects that
    the statistics of extreme values of $\log |\zeta(\tfrac 12 + \ii \tau + \ii h)|$, $|h| \leq \log^\theta T$, should be similar to the ones of $(\log T)^{1+\theta}$ Gaussian random variables
    of variance $\frac{1}{2}\log\log T$.
    If the random variables were independent, this is the so-called {\it Random Energy Model} (REM) in statistical mechanics introduced in \cite{Derrida_1981}.
    For $\theta\geq 0$, it is not hard to check, using basic Gaussian tail estimates, that the expression \eqref{eq:def:free.energy} corresponds to the free energy of the model, and the results of Theorem~\ref{thm: max}, to the maximum of the REM.
    For more on this, we refer to \cite{MR3380419}, where many techniques from REM were introduced to analyze log-correlated processes.

    The REM heuristic is of course limited as the values of $\log |\zeta(\tfrac 12 + \ii \tau + \ii h)|$, $|h| \leq \log^\theta T$, are correlated.
    In fact, they are log-correlated if $|h-h'|\leq 1$, as first noticed \cite{bourgade}. A good probabilistic model for the extreme values in the case $\theta=0$ is therefore a branching random walk. This is explained in more details in Section~\ref{sect: outline} and illustrated in Figure~\ref{fig: tree}.
    For $\theta> 0$, our results show that the correlations do not affect large values at leading order (though the proofs must take them into account).
    As argued in Section~\ref{sect: outline}, we believe that the correct probabilistic model for large values in this case is $\log^\theta T$ independent branching random walks.
    One implication is that the REM heuristic should persist to subleading order (but fail at the level of fluctuations).
    In view of this, we believe that conjecture \eqref{eqn: FHK} needs to be expanded as follows to include large intervals:

    \begin{conjecture}\label{conj}
        Let $\theta \geq 0$ be given and let $m(\theta)$ be as in \eqref{eq:def:free.energy}.
        Let $\tau$ be a random variable uniformly distributed on $[T,2T]$.
        For any $0<\delta<1$, there exists $C=C(\delta)>0$ large enough and independent of $T$, such that with probability $1-\delta$,
        \begin{equation}\label{eqn: FHK.extended}
            \begin{aligned}
                &\max_{|h| \leq \log^\theta T}  \log |\zeta(\tfrac 12 + \ii \tau + \ii h)|- \big(m(\theta) \log\log T - r(\theta) \log\log\log T\big) \in [-C,C],
            \end{aligned}
        \end{equation}
        where
        \begin{equation*}
            \text{$r(\theta) = \frac{3}{4}$ \ \ if $\theta = 0$ \qquad and \qquad $r(\theta) = \frac{1}{4 \sqrt{1 + \theta}}$ \ \ if $\theta > 0$. }
        \end{equation*}
    \end{conjecture}
    In particular, we expect a discontinuity of $r(\theta)$ as $\theta\downarrow 0$.
    An analysis of a model of the Riemann zeta function shows that the discontinuity can be resolved by approaching $0$ at a suitable rate. Namely if $\theta\sim (\log\log T)^{-\alpha}$, it is expected that $r(\theta)=\frac{1+2\alpha}{4}$, interpolating between $1/4$ and $3/4$ for $0<\alpha<1$, see \cite{ADH_2021}. Such hybrid statistics have been studied in the context of branching random walks, see \cite{MR3358969} and \cite{MR4156608}.

    For $\theta<0$, our analysis suggests that the correct model consists of a single random walk up to {\it time} $|\theta|\log\log T$ followed by a branching random walk. The maximum on such intervals would then be consistent with the level proposed in Section 2\hspace{0.5mm}(c)(ii) of \cite{MR3151088},
    \begin{equation}\label{eqn: FHK.extended.theta.smaller.0}
        \begin{aligned}
            &\max_{|h| \leq \log^\theta T}  \log |\zeta(\tfrac 12 + \ii \tau + \ii h)| \\[-2mm]
            &\qquad= m(\theta) \log\log T - \frac{3}{4} \log\log\log T+ \sqrt{\tfrac{|\theta|}{2}\log\log T}\cdot \mathcal{Z} + \OO_{\P}(1),
        \end{aligned}
    \end{equation}
    where $\mathcal Z$ is a standard Gaussian random variable. As explained in  Section~\ref{sect: outline}, the additional fluctuation would represent the contribution of the Dirichlet polynomial $\sum_{\log p\leq \log^{|\theta|} T}\Re p^{-1/2-\ii(\tau+h)}$ which is essentially the same random variable for all $h$'s in the interval $|h|\leq \log^\theta T$.

    \subsection{Relations to other models}

    When $-1 < \theta \leq 0$, Conjecture~\ref{conj} is based on modelling $\zeta$ by the characteristic polynomial of a random unitary matrix (CUE).
    More precisely, if $M_N$ is a random matrix sampled from the Haar measure on the unitary group $\mathcal U(N)$, one can consider the moments
    \begin{equation}
        \E\bigg[\Big(\frac{1}{2\pi}\int_{0}^{2\pi}|\det(\mathbb{I} - e^{-\ii h}M_N)|^{2\beta}\rd h\Big)^k\bigg], \quad k > 0, ~\beta > 0.
    \end{equation}
    These can be computed in the limit $N\to\infty$, at least heuristically, using Selberg integrals and the Fisher-Hartwig formula, cf.\ \cite{MR3151088}.
    Exact expressions were recently obtained in \cite{bailey-keating} in the regime $k,\beta\in \N$.
    The statistics of $\log \int_{0}^{2\pi}|\det(\mathbb{I} - e^{-\ii h}M_N)|^{2\beta}\rd h$ and of $\max_{h\in[0,2\pi]} |\det(\mathbb{I} - e^{-\ii h}M_N)|$ in the limit $N\to\infty$ can be inferred from the asymptotics of the moments
    by comparison with log-correlated processes, cf.\ \cite{fyodorov-gnutzmann-keating} for a numerical study.
    In the CUE setting, the freezing analogue of \eqref{eqn: FK} and the leading order as in \eqref{eqn: FHK} were proved in \cite{ABB_2017}.
    The subleading order of the maximum was proved in \cite{paquette-zeitouni}, and up to constant $C$ in \cite{chhaibi-madaule-najnudel}.

    From the analysis of a particular variant of the log-correlated REM model, \cite{fyodorov-bouchaud} conjectured an exact formula for the density of the total mass of the sub-critical Gaussian multiplicative chaos (GMC) measure associated to the Gaussian free field (GFF) on the unit circle, cf.\ \cite{rhodes-vargas}.
    In the critical case, they conjectured that the fluctuations of the maximum can be captured by a sum of two Gumbel variables.
    Both results were proved in \cite{remy}.
    Naturally, these results are expected to hold in the CUE setting, where the GMC measure is the limit of
    \begin{equation}
        \frac{|\det(\mathbb{I} - e^{-\ii h}M_N)|^{2\beta}}{\E[|\det(\mathbb{I} - e^{-\ii h}M_N)|^{2\beta}]}\frac{\rd h}{2\pi},
    \end{equation}
    as proved by \cite{webb2015} when $-1/4 < \beta < 1/\sqrt{2}$, and by \cite{nikula-saksman-webb} when $1/\sqrt{2} \leq \beta < 1$.
    Such a random measure can also be considered in the context of the Riemann zeta function for mesoscopic intervals of length $\log^\theta T$, $-1<\theta\leq 0$, with $|\zeta(\tfrac 12 + \ii \tau + \ii h)|$ in place of $|\det(\mathbb{I} - e^{-\ii h}M_N)|$.
    (There does not seem to be any obvious equivalent for macroscopic intervals, $\theta>0$, in the CUE model.)
    A step in this direction was made in \cite{saksman-webb} where $\zeta(\tfrac 12 + \ii \tau + \ii h)$, $h\in \R$, was shown to converge, as $T\to\infty$, when considered as a random variable on the space of tempered distributions.

    Another model for the large values of $\log|\zeta(\tfrac 12 + \ii \tau + \ii h)|$, $h\in[-1,1]$, is to consider a random Dirichlet polynomial $X_h = \Re \sum_{p\leq T} \ p^{- 1/2 -\ii h}U_p$, where $(U_p, \, p \text{ primes})$ are i.i.d.\ uniform random variables on the unit circle, cf. \cite{harper_note_13, MR3619786, MR3906393}.
    The analogue of conjecture \eqref{eqn: FHK} for this model was proved up to second-order corrections in \cite{MR3619786}, and large deviations and continuity estimates for the derivative were found in \cite{MR3906393}.
    The limit of the corresponding multiplicative chaos measure was obtained in \cite{arXiv:1604.08378,saksman-webb}.
    A proof of the freezing phase transition was given in \cite{arguin-tai}.
    In the latter, the limit of the Gibbs measure $\exp(\beta X_h)\rd h$
    is also studied in the supercritical regime $\beta > 2$, showing that it is supported on $h$'s that are at a relative distance of order one or order $(\log T)^{-1}$ of each other.
    This result was used in \cite{ouimet2018} to prove that the normalized Gibbs weights converge to a Poisson-Dirichlet distribution.

    \vspace{2mm}
    \begin{notation}\rm
        For the rest of the paper, $\tau$ denotes a uniform random variable on $[T,2T]$.
        For any event $A_T\subseteq [T,2T]$ and a random variable $X_T:[T,2T]\rightarrow \C$, we write
        \begin{equation*}
            \P(A_T) = \frac{1}{T} \m(A_T) \quad \text{and} \quad \E[X_T] = \frac{1}{T} \int_T^{2T} X_T(t) \rd t.
        \end{equation*}
        We also use the standard $\oo$ and $\OO$ notations: thus, $f(T)=\oo(g(T))$ if $|f(T)/g(T)|$ tends to $0$ as $T\to\infty$ when the parameters $\theta$, $\beta$ and $\e$ are fixed.
        Similarly, we write $f(T)=\OO(g(T))$ if $\limsup |f(T)/g(T)|$ is bounded for $\theta$, $\beta$ and $\e$ fixed.
        We sometimes write for conciseness $f(T)\ll g(T)$ if $f(T)=\OO(g(T))$, and also $f(T) \asymp g(T)$ if both $f(T) \ll g(T)$ and $g(T) \ll f(T)$ hold.
        In some statements, we write $f(T)\ll_A g(T)$ or $f(T)=\OO_A(g(T))$ to highlight the dependence on a specific parameter $A$ in the implicit constant.
        In some of the proofs, we use the common convention that $\varepsilon$ denotes an arbitrarily small positive quantity that may vary from line to line.
        We will also encounter some arithmetical functions familiar in number theory.
        These include: $\omega(n)$ (which counts the number of distinct primes dividing $n$), $\Omega(n)$ (which counts with multiplicity the number of primes dividing $n$), and the M\"obius function $\mu(n)$ (which equals $0$ if $n$ is divisible by the square of a prime, and equals $(-1)^{\omega(n)}$ if $n$ is square-free).
        Throughout the paper, $x \vee y$ and $x \wedge y$ refer to $\max\{x,y\}$ and $\min\{x,y\}$, respectively.
    \end{notation}

    \subsection{Outline of the proof}\label{sect: outline}

	For $\theta > 0$, the upper bound part of Theorem~\ref{thm: freezing} and Theorem~\ref{thm: max} follows from the moment estimates
	\begin{equation}\label{eq:moments}
		\E\Big[|\zeta(\tfrac 12 + \ii \tau)|^{\beta}\Big] \ll (\log T)^{\beta^2 / 4 + \e},
	\end{equation}
	and from a discretization result which roughly shows that for a Dirichlet polynomial $D$ that approximates zeta, and for $\beta \geq 1$, we have
	\begin{equation}\label{eq:discretize}
        \max_{|h| \leq \log^\theta T} |D(\tfrac 12 + \ii \tau + \ii h)|^{\beta} \ll \sum_{|k| \leq \log^{1 + \theta} T} \big | D \big ( \tfrac{1}{2} + \ii \tau + \tfrac{2\pi \ii k}{\log T} \big ) \big |^{\beta}.
	\end{equation}
    Equation~\eqref{eq:discretize} tells us that the process $(\zeta(\tfrac{1}{2} + \ii \tau + \ii h), ~|h| \leq \log^\theta T)$ varies on a $(\log T)^{-1}$ scale, so that the maximum and moments of $\log |\zeta|$ on an interval of length $\OO(\log^\theta T)$ behave as those of $\OO(\log^{1 + \theta}T)$ i.i.d.~Gaussian random variables of variance $\tfrac{1}{2} \log \log T$.%
    \footnote{As in the branching random walk setting, the log-correlations are important in the proof of the first-order asymptotics of the maximum, high points and moments, but they do not influence the results. When comparing Gaussian fields, Slepian's lemma tells us that, at equal variance, the field with no correlations will have, on average, the highest maximum and the highest number of points above any fixed proportion of the maximum (the asymptotics of the moments are derived directly from these two quantities). Therefore, the asymptotics of the maximum and moments for i.i.d.\ Gaussians are always an upper bound for those of log-correlated Gaussian fields. It turns out that we get a matching lower bound by a coarse-graining of the scales following \cite{MR3380419}.
    This is why our heuristic here is phrased in terms of i.i.d.\ Gaussians, because the correlations ultimately only matters for the proof, not the actual results.}
    The limitation to $\theta \leq 3$ comes from the fact that the upper bounds \eqref{eq:moments} are not known unconditionally for $\beta > 4$.
		
    When $\theta < 0$, the upper bounds in Theorem~\ref{thm: freezing} and Theorem~\ref{thm: max} are a bit more delicate.
    We follow essentially the same strategy, but we apply it to the function
	\begin{equation}\label{eq:newzeta}
        (\zeta \cdot e^{-\mathcal{P}_{|\theta|}})(\tfrac 12 + \ii \tau), \quad \text{where } \mathcal{P}_{\alpha}(s) =\hspace{-2mm}\sum_{\log p \leq \log^{\alpha} T} \frac{1}{p^s} ~~~\text{for } \alpha > 0,
	\end{equation}
    instead of $\zeta(\tfrac 12 + \ii \tau)$.
    The reason is that, when $\theta < 0$, the contribution of the primes up to scale $|\theta|$ is negligible with high probability.
    Namely, with probability $1 -\oo(1)$,
	\begin{equation}
		\max_{|h| \leq \log^\theta T} \big|\mathcal{P}_{|\theta|}(\tfrac 12 + \ii \tau+\ii h)\big| = \oo(\log\log T).
	\end{equation}
	When $\tau$ is restricted to a specific event $\mathcal{A}(T)$ on which \eqref{eq:newzeta} can be discretized as in \eqref{eq:discretize}, we can show that
	\begin{equation}
        \E\Big[\big|(\zeta \cdot e^{-\mathcal{P}_{|\theta|}})(\tfrac{1}{2} + \ii \tau)\big|^{\beta}\Big] \ll (\log T)^{(\beta^2 / 4) \cdot (1 + \theta) + \e},
	\end{equation}
	for $\beta \leq 2$. This explains the additional factor $(\beta^2/4) \theta$ in $f_{\theta}(\beta)$ when $-1 < \theta < 0$ and $\beta \leq 2$.
	
    We then turn to the lower bound part of Theorem~\ref{thm: freezing} and Theorem~\ref{thm: max}.
    The lower bounds in Theorem~\ref{thm: max} follow directly from Theorem~\ref{thm: freezing} (see \eqref{eq:prop:lower.bound.zeta.maximum.eq.1}), so it is enough to discuss Theorem~\ref{thm: freezing}.
	
    The problem is first reduced to obtaining lower bounds for moments off the critical line.
    In particular, it is shown, uniformly in $\tfrac 12 \leq \sigma \leq \tfrac 12 + (\log T)^{\theta - \e}$ and for any given $\e > 0$, that with probability $1-\oo(1)$,
    \begin{equation}
        \int_{-\log^\theta T}^{\log^\theta T} |\zeta(\sigma + \ii \tau + \ii h)|^{\beta} \rd h \ll \int_{-3 \log^\theta T}^{3 \log^\theta T} |\zeta(\tfrac 12 + \ii \tau + \ii h)|^{\beta} \rd h + \frac{1}{(\log T)^{96}}.
	\end{equation}
    This is accomplished using a result of \cite{gabriel_1927} for subharmonic functions, and the construction of an explicit entire function which is a good approximation to the indicator function of the rectangle $\mathcal{R} = \{ \sigma + \ii u : |u| \leq (\log  T)^{\theta}, \, \tfrac 12 \leq \sigma \leq \tfrac 12 + (\log T)^{\theta - \e}\}$ in the whole strip $\tfrac 12 \leq \Re s$. The fact that the interval can be very small when $\theta < 0$ makes this part rather technical. We believe that this result might be useful in other applications as well.
	
	The problem is therefore reduced to obtaining a good lower bound for
	\begin{equation}
        \int_{-\log^\theta T}^{\log^\theta T} |\zeta(\sigma_0 + \ii \tau + \ii h)|^{\beta} \rd h, \quad \text{with } \sigma_0 = \frac{1}{2} + \frac{1}{(\log T)^{1 - \delta}},
	\end{equation}
    for some sufficiently small $\delta > 0$.
    We adapt mollification results from \cite{ABBRS_2019} to show that, outside of an event of probability $\oo(1)$, the problem can be reduced to understanding
    \vspace{-1.5mm}
	\begin{equation}
        \int_{-\log^\theta T}^{\log^\theta T} \exp\big(\beta \, \Re \mathcal{P}_{1 - \delta}(\sigma_0 + \ii \tau + \ii h)\big) \rd h.
	\end{equation}

	The proof of the lower bound is now restricted to the problem of understanding the correlation structure of the process
    \begin{equation}\label{eq:process}
    	\big(\Re \mathcal{P}_{1 - \delta}(\sigma_0 + \ii \tau + \ii h), ~ |h| \leq \log^\theta T\big).
    \end{equation}
    The remaining part of the argument is done in Section~\ref{sect: lower proofs} by a multiscale second moment method introduced in \cite{MR3380419}.	
    The covariance of the process \eqref{eq:process} can be computed using Lemma~\ref{lem: moments} with $a(p)=p^{-\sigma_0}(p^{-\ii h}+p^{-\ii h'})$:
    \begin{equation}\label{eqn: correlation}
        \begin{aligned}
            &\E\Big[\Re \mathcal{P}_{1 - \delta}(\sigma_0 + \ii \tau + \ii h)\cdot \Re \mathcal{P}_{1 - \delta}(\sigma_0 + \ii \tau + \ii h')\Big] \\[1mm]
            &\hspace{20mm}=\frac{1}{2}\sum_{\log p \leq (\log T)^{1 - \delta}} \frac{\cos(|h-h'|\log p)}{p^{2\sigma_0}} + \OO(1).
        \end{aligned}
    \end{equation}
    The cosine factor implies that primes smaller than $\exp(|h-h'|^{-1})$ are almost perfectly correlated, whereas primes greater than  $\exp(|h-h'|^{-1})$ decorrelate quickly.
    In fact, the covariance can be evaluated precisely using the prime number theorem and equals $\frac{1}{2}\log |h-h'|^{-1} +\OO(1)$.
    This shows that the process is approximatively a log-correlated Gaussian process.
    (This is also true for $\log|\zeta|$ in the sense of finite-dimensional distributions as shown in \cite{bourgade}.)

    The identification with a log-correlated process is useful as it suggests that the Dirichlet polynomials have an underlying tree structure.
    To see this, consider the {\it increments}
    \begin{equation}
        P_k(h) = \sum_{e^{k-1} < \log p \leq e^{k}} \Re \frac{1}{p^{\sigma_0 + \ii \tau + \ii h}}, \quad 1 \leq k \leq \log\log T.
    \end{equation}
    The range of primes is chosen so that each $P_k$ has variance $\tfrac 12 +\oo(1)$.
    In this framework, the Dirichlet polynomial at $h$ can be seen as a random walk with independent and identically distributed increments.
    However, the random walks for different $h$'s are not independent by \eqref{eqn: correlation}.
    In fact, the walks are almost perfectly correlated until they {\it branch out} around the prime $p\approx \exp(|h-h'|^{-1})$, corresponding to the increment $k(h,h')=\log |h-h'|^{-1}$.
    Since $k$ goes to essentially $\log\log T$, the analysis can be restricted to $h$'s on a grid with mesh $(\log T)^{-1}$.
    Furthermore, the $h$'s in an interval of size $(\log T)^{-\alpha}$, for $0<\alpha<1$, will share the same increments up to $k\approx \alpha \log\log T$.

    The above observations have important consequences for the probabilistic analysis.
    For $\theta=0$, this means that the process \eqref{eq:process} on an interval of order one is well approximated by a Gaussian process indexed by a tree of average degree $e=2.718\dots$,
    where the independent increments $P_k(h)$ are identified with the edges of the tree. Note that the number of leaves on the interval $[-1,1]$ is then $\approx e^{\log\log T}=\log T$.
    Equivalently, the walks $\sum_k P_k(h)$, $h\in[-1,1]$, can be seen as a branching random walk on a Galton-Watson tree with an average number of offspring $e$, cf. Figure~\ref{fig: tree}.

    \begin{figure}[ht]
        \hspace{-3mm}\includegraphics[width=0.8\linewidth,height=0.7\linewidth]{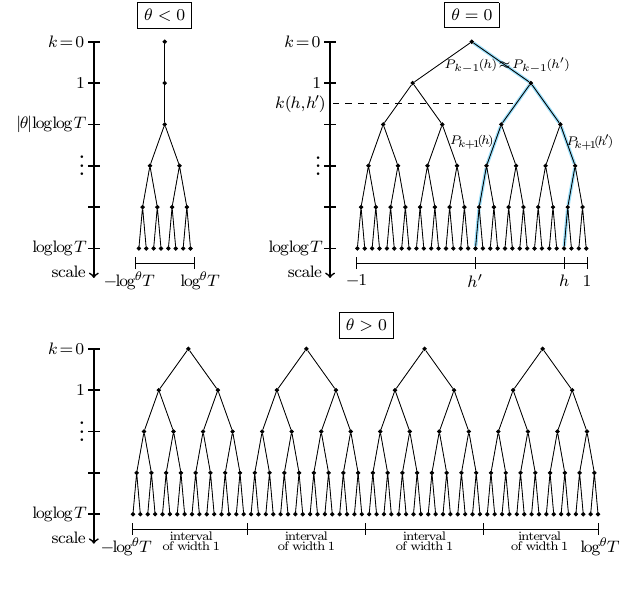}
        \vspace{-6mm}
        \caption{An illustration of the branching structure of $\sum_{k}P_k$ for the interval $[-\log^{\theta} T, \log^{\theta} T]$ with $\theta < 0$, $\theta = 0$ and $\theta > 0$. When $\theta < 0$, the increments $P_k$ approximately coincide at each scale up to scale $|\theta| \log \log T$ (so one branch), followed by the structure of a branching random walk. When $\theta = 0$, we have one branching random walk starting at scale $0$.
        When $\theta > 0$, there is a forest of $\asymp \log^{\theta} T$ approximately independent branching random walks.} \vspace{-3mm}
        \label{fig: tree}
    \end{figure}

    When $\theta<0$, the tree structure suggests that the primes up to $\exp(\log^{|\theta|} T)$ do not contribute to large values, since they should be essentially the same for all $h$'s in the interval .
    Therefore these primes can be cutoff at a low cost, cf.\ Corollary~\ref{cor:DP.cutting.0.to.theta}.
    This is equivalent to restricting to a subtree of the one on $[-1,1]$ with $(1+\theta)\log\log T$ increments and $\log^{1+\theta} T$ leaves, yielding a maximum at leading order of $(1+\theta)\log\log T$ by the REM heuristic.

    The case $\theta>0$ stands out as the analogy with branching random walks fails.
    This is because the random walks for $h$ and $h'$ are essentially independent when ${|h-h'|>1}$.
    Therefore the right probabilistic model seems to consist of $\asymp \log^\theta T$ independent branching random walks corresponding to different intervals of order one, see Figure~\ref{fig: tree}.
    A large class of similar models (called CREM's for {\it Continuous Random Energy Models}) have been studied in \cite{Bovier-Kurkova2004}, see \cite{Bovier2006,Bovier2017} for a review.
    It turns out that the large values at leading order correspond to the ones of a REM with $\log^{1+\theta} T$ variables of variance $\frac{1}{2}\log\log T$.
    This yields a maximum of $\sqrt{1+\theta}\log\log T$ at leading order.
    In fact, in view of the extreme value statistics of CREM's, we expect that the REM heuristic holds for subleading corrections.
    This is the motivation for Conjecture~\ref{conj}.

\section{Upper bounds}\label{sec:upper.bounds}

    \subsection{Moment estimates}

    We will need a number of moment estimates which we state below.

    \begin{proposition}\label{pr:soundbound}
        Assume the Riemann hypothesis. Let $\beta > 0$ and $\e > 0$ be given. Then,
        \begin{equation}
            \E\big[|\zeta(\tfrac 12 + \ii \tau)|^{\beta}\big] \ll (\log T)^{\beta^2 / 4 + \e}.
        \end{equation}
    \end{proposition}

    \begin{proof}
        See Corollary A in \cite{Sound09}.
    \end{proof}

    \begin{proposition}\label{pr:maksbound}
        Let $0 < \beta \leq 4$ be given. Then,
        \begin{equation}
            \E\big[|\zeta(\tfrac 12 + \ii \tau)|^{\beta}\big] \ll (\log T)^{\beta^2 / 4}.
        \end{equation}
    \end{proposition}

    \begin{proof}
        See Theorem 1 in \cite{heap-radziwill-soundararajan}.
    \end{proof}

    The proof of Proposition~\ref{pr:soundbound} is based on the following deterministic upper bound for
    $\zeta$: Suppose that $T$ is large. Let $T \leq t \leq 2T$, and let $2 \leq x \leq T^2$.
    Then, as $T\to\infty$, we have
    \begin{equation}\label{eq:prop:Soundararajan.2009.Proposition.and.Lemma.2}
        \log |\zeta(\tfrac{1}{2} + \ii t)| \leq \mathrm{Re} \, \sum_{p \leq x} \frac{1}{p^{\frac{1}{2} + \frac{1}{\log x} + \ii t}} \frac{\log(x / p)}{\log x} +  \frac{\log T}{\log x} + \OO(\log \log \log T),
    \end{equation}
    see Proposition and Lemma 2 in \cite{Sound09}.
    On the Riemann hypothesis, the upper bounds in Theorem~\ref{thm: freezing} and Theorem~\ref{thm: max} could be proved in a simpler way by using this deterministic bound, and by proving the corresponding results for the Dirichlet polynomials. For unconditional results, such a deterministic upper bound is not available. We need to work {\it on average} to discard the contribution of large primes. This is the purpose of
    Lemmas~\ref{le:1nd},~\ref{le:2nd},~\ref {le:3rd} and Proposition~\ref{pr:momenttwi} below.

    \vspace{3mm}
    In order to compute the moments of $\zeta\cdot e^{-\mathcal P_{|\theta|}}$, we will need to express $e^{-\mathcal P_{|\theta|}}$ as a finite Dirichlet polynomial.
    To this aim, notice that if $|z|\leq \nu/10$ for some $\nu\in \N$, we have
    $
    \big|e^z - \sum_{j=0}^{\nu} \frac{z^j}{j!}\big| \leq e^{-\nu}.
    $
    Consider more generally $e^{\lambda\mathcal P(s)}$ with $\lambda\in \C$ and $\mathcal P(s)=\sum_{p\leq X}a(p)p^{-s}$ for some completely multiplicative function $a$.
    We have by the above, assuming $|\lambda\mathcal P(s)|\leq \nu/10$ for some $\nu\in \N$, and by the multinomial formula, that
    \vspace{1mm}
    \begin{equation}\label{eqn: truncation}
        \bigg|e^{\lambda\mathcal P(s)} - \sum_{k=0}^\nu \frac{\lambda^k}{k!}\Big(\sum_{p\leq X} \frac{a(p)}{p^s}\Big)^k\bigg|
        = \bigg|e^{\lambda\mathcal P(s)} - \hspace{-2mm}\sum_{\substack{\Omega(n)\leq \nu\\ p\mid n\Longrightarrow p\leq X}}\hspace{-2mm}\frac{\lambda^{\Omega(n)}a(n)\mathfrak g(n)}{n^s}\bigg| \leq e^{-\nu},
    \end{equation}
    where $\Omega(n)$ is the number of prime factors of $n$ with multiplicity.
    Here, $\mathfrak g$ is the multiplicative function defined by $\mathfrak g(p^k)=1/k!$ for all integers $k$ and primes $p$.

    \newpage
    The relevant function $a$ for $e^{-\mathcal P_{|\theta|}}$ will be of the following form:
    Given $\alpha, \beta \in \mathbb{R}$ and $\theta > -1$, let $\mathfrak{F}_{\alpha, \beta, \theta}(n)$ denote a completely multiplicative function such that
    \begin{equation}\label{eq:multiplicative.function}
        \mathfrak{F}_{\alpha, \beta, \theta}(p) \leqdef
        \begin{cases}
            \alpha, &\mbox{if } \log p \leq \log^{|\theta|} T, \\[1mm]
            \beta, &\mbox{if } \log^{|\theta|} T < \log p.
        \end{cases}
    \end{equation}

    In the next three lemmas, we control various terms with the aim of proving the moment estimate in Proposition~\ref{pr:momenttwi}, which we will need in the case of short intervals.

	\begin{lemma}\label{le:1nd}
		Let $- 1 < \theta < 0$, $\beta > 0$ and $\e > 0$ be given.
        Then,
		\begin{equation}\label{eq:le:1nd}
            \E\bigg[\Big|\sum_{\substack{\Omega(n) \leq 100 \lfloor \log\log T \rfloor \\ p | n \implies \log p \leq \log^{1 - \e} T}} \frac{\mathfrak{F}_{0, \beta / 2, \theta}(n) \mathfrak{g}(n)}{n^{1/2 + \ii \tau}}\Big|^2\bigg] \ll (\log T)^{\beta^2 ( 1 + \theta) / 4}.
		\end{equation}
	\end{lemma}

	\begin{proof}
        Notice that the Dirichlet polynomial in \eqref{eq:le:1nd} has length $ \ll T^{\delta}$ for any fixed $\delta > 0$. In particular, by the mean-value formula (Lemma~\ref{lem:ABBRS.2018.Lemma.3.3}),
		\begin{equation*}
            \E\bigg[\Big|\sum_{\substack{\Omega(n) \leq 100 \lfloor \log\log T \rfloor \\ p | n \implies \log p \leq \log^{1 - \e} T}} \frac{\mathfrak{F}_{0, \beta / 2, \theta}(n) \mathfrak{g}(n)}{n^{1/2 + \ii \tau}}\Big|^2\bigg] \ll \sum_{\substack{\Omega(n) \leq 100 \lfloor \log\log T \rfloor \\ p | n \implies \log p \leq \log^{1 - \e} T}} \frac{\mathfrak{F}_{0, \beta / 2, \theta}(n)^2 \mathfrak{g}(n)^2}{n}.
		\end{equation*}
        Dropping the restriction on $\Omega(n)$ and expressing the sum as an Euler product yield
        \begin{equation}
            \sum_{\substack{ p | n \implies \log p \leq \log^{1 - \e} T}} \frac{\mathfrak{F}_{0, \beta / 2, \theta}(n)^2 \mathfrak{g}(n)^2}{n}
            = \prod_{\log p \leq \log^{1 - \e} T} \bigg(1 + \sum_{k\geq 1} \frac{\mathfrak{F}_{0,\beta/2,\theta}(p)^{2k} \mathfrak{g}(p^k)^2}{p^k}\bigg).
        \end{equation}
        The logarithm of the right-hand side is easily evaluated using the prime number theorem (see Lemma~\ref{lem:PNT.estimates}) and is $ (\beta^2 (1 + \theta)/4) \log\log T +\OO(1)$. This proves the claimed bound.
 	\end{proof}

    \begin{lemma}\label{le:2nd}
    	Let $-1 < \theta < 0$, $0 < \beta \leq 2$ and $\e > 0$ be given. Then,
    	\begin{equation}\label{eq:le:2nd}
            \E\bigg[|\zeta(\tfrac 12 + \ii \tau)|^2 \cdot \Big | \sum_{\substack{\Omega(n) \leq 100 \lfloor \log\log T \rfloor \\ p | n \implies \log p \leq \log^{1 - \e} T}} \frac{\mathfrak{F}_{-1,\beta / 2 - 1, \theta}(n)\mathfrak{g}(n)}{n^{1/2 + \ii \tau}} \Big |^2\bigg] \ll (\log T)^{\beta^2 ( 1 + \theta) / 4 + \e}.
    	\end{equation}
    \end{lemma}

    \begin{proof}
    	By Theorem 1 in \cite{BettinChandeeRadziwill}, the left-hand side of \eqref{eq:le:2nd} is
    	\begin{equation}\label{eq:le:2nd.beginning}
            \begin{aligned}
                &\leq \frac{1}{T} \int_{\R} \sum_{\substack{\Omega(n) \leq 100 \lfloor \log \log T \rfloor \\ \Omega(m) \leq 100 \lfloor \log\log T \rfloor \\ p | n \implies \log p \leq \log^{1 - \e} T \\ p | m \implies \log p \leq \log^{1 - \e} T}} \hspace{-7mm}\frac{\mathfrak{F}_{-1, \beta / 2 - 1, \theta}(n m)\mathfrak{g}(n) \mathfrak{g}(m)}{[n,m]} \, \log \Big ( \frac{(n,m)^2}{2\pi n m} \Big ) \Phi \Big ( \frac{t}{T} \Big ) \rd t \\
                &\quad + \frac{1}{T} \int_{\R} \sum_{\substack{\Omega(n) \leq 100 \lfloor \log \log T \rfloor \\ \Omega(m) \leq 100 \lfloor \log\log T \rfloor \\ p | n \implies \log p \leq \log^{1 - \e} T \\ p | m \implies \log p \leq \log^{1 - \e} T}} \hspace{-7mm}\frac{\mathfrak{F}_{-1, \beta / 2 - 1, \theta}(n m)\mathfrak{g}(n) \mathfrak{g}(m)}{[n,m]} \big(\log t + 2 \gamma\big ) \Phi \Big ( \frac{t}{T} \Big ) \rd t
                + \OO(T^{- \e}),
            \end{aligned}
    	\end{equation}
        where $\Phi$ is a smooth non-negative function such that $\Phi(x) \geq 1$ for all $1 \leq x \leq 2$, with support contained in say $[0,3]$, and $(n,m)$ and $[n,m]$ stand for the greatest common divisor and the least common multiple, respectively.

        We first note that if $n,m$ have the prime factorization $n = \prod_{i=1}^r p_i^{\alpha_i}$ and $m = \prod_{i=1}^r p_i^{\beta_i}$, where the $\alpha_i$'s and $\beta_i$'s are possibly $0$, then $[n,m] = \prod_{i=1}^r p_i^{\alpha_i \vee \beta_i}$.
        This means that if $a(n)$ and $b(m)$ are two bounded multiplicative functions, we have
        \begin{equation}\label{eqn: pairs}
            \sum_{\substack{p | n \implies p \leq X \\ p | m \implies p \leq X}} \frac{a(n)b(m)}{[n,m]} = \prod_{p \leq X} \bigg(1 + \sum_{k\geq 1} p^{-k} \sum_{i,j:\max(i,j) = k} a(p^i) b(p^j)\bigg).
        \end{equation}

        Using Chernoff's bound, we can get rid of the restriction $\Omega(n) \leq 100 \lfloor \log\log T \rfloor$ in \eqref{eq:le:2nd.beginning}.
        It suffices to notice that the contribution of each sum over $n$ with $\Omega(n) > 100 \lfloor \log\log T \rfloor$ is
    	\begin{equation}
            \begin{aligned}
                &\ll \log T \sum_{\substack{p | n \implies p \leq T \\ p | m \implies p \leq  T}} \frac{|\mathfrak{F}_{-1, \beta / 2 - 1, \theta}(nm)| \mathfrak{g}(n) \mathfrak{g}(m)}{[n,m]} e^{ \Omega(n) - 100 \log\log T} \\
                &\ll (\log T)^{-99} \prod_{p \leq T} \Big ( 1 + \frac{(1 + e) |\mathfrak{F}_{-1, \beta / 2 - 1, \theta}(p)|}{p} + \frac{e\cdot |\mathfrak{F}_{-1, \beta/2 - 1, \theta}(p)|^2}{p}\Big ) \\[2mm]
                &\ll (\log T)^{-99 } \cdot (\log T)^{1 + 2 e} = \oo(1),
            \end{aligned}
    	\end{equation}
        where we used \eqref{eqn: pairs} with $a(n) = |\mathfrak{F}_{-1, \beta / 2 - 1, \theta}(n)| \mathfrak{g}(n) e^{\Omega(n)}$, $b(m) = |\mathfrak{F}_{-1, \beta / 2 - 1, \theta}(m)| \mathfrak{g}(m)$.
        The contribution of each sum over $m$ with $\Omega(m) > 100 \lfloor \log\log T \rfloor$ can be removed in the same manner.

        Considering the sums in \eqref{eq:le:2nd.beginning} without the restriction on $\Omega(n)$ and $\Omega(m)$, we get by \eqref{eqn: pairs} and Lemma~\ref{lem:PNT.estimates},
    	\begin{equation}\label{eq:le:2nd.middle}
            \begin{aligned}
                &\sum_{\substack{p | n \implies \log p \leq \log^{1 - \e} T \\ p | m \implies \log p \leq \log^{1 - \e} T}} \frac{\mathfrak{F}_{-1, \beta / 2 - 1, \theta}(n m)\mathfrak{g}(n) \mathfrak{g}(m)}{[n,m]} \\[-1mm]
                &\hspace{10mm}\asymp \prod_{\log p \leq \log^{1 - \e} T} \Big ( 1 + \frac{2 \mathfrak{F}_{-1, \beta / 2 - 1, \theta}(p) + \mathfrak{F}_{-1, \beta / 2 - 1, \theta}(p)^2}{p} \Big ) \\
                &\hspace{10mm}\asymp (\log T)^{-|\theta|} \cdot (\log T)^{(\beta^2 / 4 - 1) \cdot (1 + \theta - \e)} \\[1mm]
                &\hspace{10mm}\ll (\log T)^{\beta^2 (1 + \theta)/4 - 1 + \e}.
            \end{aligned}
    	\end{equation}
        In particular, this means that the second integral in \eqref{eq:le:2nd.beginning} is $\ll (\log T)^{\beta^2 (1 + \theta)/4 + \e}$.
    	
    	To evaluate the first integral in \eqref{eq:le:2nd.beginning} , write
        \begin{equation}
            \log \Big ( \frac{(m,n)^2}{m n} \Big ) = \frac{1}{2\pi \ii} \oint_{|z| = 1 / \log T} \Big ( \frac{(m,n)^2}{m n} \Big )^{z} \cdot \frac{\rd z}{z^2}.
        \end{equation}
        Then, we end up having to evaluate
        \begin{equation}\label{eq:finalizer}
            \frac{1}{2\pi \ii} \oint_{|z| = 1 / \log T} \sum_{\substack{ p | n \implies \log p \leq \log^{1 - \varepsilon} T \\ p | m \implies \log p \leq \log^{1 - \varepsilon} T}} \frac{\mathfrak{F}_{-1, \beta / 2 - 1, \theta}(m n) \mathfrak{g}(m) \mathfrak{g}(n)}{[m, n]} \, \Big ( \frac{(m,n)^2}{m n} \Big )^{z} \cdot \frac{\rd z}{z^2}.
        \end{equation}
        As above, the sum over $m$ and $n$ factors into an Euler product which is
        \begin{equation}\label{eq:bounder}
            = \prod_{\log p \leq \log^{1 - \varepsilon} T} \left(
            \begin{array}{l}
                1 + \frac{2 \mathfrak{F}_{-1,\beta/2 - 1,\theta}(p)}{p^{1 + z}} + \frac{\mathfrak{F}_{-1,\beta/2 - 1,\theta}(p)^2}{p} \\[1mm]
                + \sum_{\substack{\scriptscriptstyle i,j\geq 0 \, : \\ \scriptscriptstyle (i \vee j) \geq 2}} \frac{\mathfrak{F}_{-1,\beta/2 - 1,\theta}(p)^{i+j}}{i! j!} \frac{p^{(i \wedge j) z}}{p^{(i \vee j) (1 + z)}}
            \end{array}
            \right).
        \end{equation}
        For $|z| = 1/\log T$, note that
        \begin{equation}
            \Bigg|\sum_{\substack{i,j\geq 0 \, :\\ (i \vee j) \geq 2}} \frac{\mathfrak{F}_{-1,\beta/2 - 1,\theta}(p)^{i+j}}{i! j!} \frac{p^{(i \wedge j) z}}{p^{(i \vee j) (1 + z)}}\Bigg| \leq \frac{1}{p^2} \sum_{i,j\geq 0} \frac{1}{i! j!} \leq \frac{e^2}{p^2},
        \end{equation}
        and a Taylor expansion yields
        \begin{equation}\label{eq:bounder.next.Taylor}
            \begin{aligned}
                &\sum_{\log p \leq \log^{1 - \varepsilon} T} \hspace{-1mm} \frac{\mathfrak{F}_{-1, \beta / 2 - 1, \theta}(p)}{p^{1 + z}} \\
                &\qquad= \sum_{\log p \leq \log^{1 - \varepsilon} T} \hspace{-1mm} \frac{\mathfrak{F}_{-1, \beta / 2 - 1, \theta}(p)}{p} \, + \, \OO\bigg(\frac{1}{\log T} \sum_{\log p \leq \log^{1 - \varepsilon} T} \frac{\log p}{p} \bigg).
            \end{aligned}
        \end{equation}
        Since the error term in \eqref{eq:bounder.next.Taylor} is $\oo(1)$ by Lemma~\ref{lem:PNT.estimates}, the Euler product in \eqref{eq:bounder} is
        \begin{equation}
            \begin{aligned}
                &\asymp \prod_{\log p \leq \log^{1 - \varepsilon} T} \Big ( 1 + \frac{2 \mathfrak{F}_{-1, \beta / 2 - 1, \theta}(p) + \mathfrak{F}_{-1, \beta / 2 - 1, \theta}(p)^2}{p} + \OO(p^{-2})\Big ) \\
                &\ll ( \log T )^{\beta^2 (1 + \theta) / 4 - 1 + \varepsilon}.
            \end{aligned}
        \end{equation}
        By putting this estimate back in the contour integral and using a trivial bound on $z^{-2}$, Equation~\eqref{eq:finalizer} is $\ll (\log T)^{\beta^2 (1 + \theta) / 4 + \varepsilon}$ as required.
    \end{proof}

    \begin{lemma}\label{le:3rd}
    	Let $\e > 0$ be given. For $\ell = 2 \lfloor \log\log T \rfloor$, we have
		\begin{equation}\label{eq:le:3rd.eq.1}
            \E\bigg[|\zeta(\tfrac 12 + \ii \tau)|^2 \cdot \Big | \frac{\mathcal{P}_{1-\e}(\tfrac{1}{2} + \ii \tau)}{5 \log\log T} \Big |^{2\ell}\bigg] \ll (\log T)^{-2},
		\end{equation}
		and
		\begin{equation}\label{eq:le:3rd.eq.2}
			\E\bigg[\Big | \frac{\mathcal{P}_{1-\e}(\tfrac{1}{2} + \ii \tau)}{5 \log\log T} \Big |^{2\ell}\bigg] \ll (\log T)^{-4}.
		\end{equation}
	\end{lemma}

	\begin{proof}
        First, we apply a moment estimate (Lemma~\ref{lem:moment.estimates.Lemma.3.Sound.2009.generalisation}) followed by a prime number theorem estimate (Lemma~\ref{lem:PNT.estimates}) to obtain
		\begin{equation}\label{eq:le:3rd.eq.1.in.proof}
            \E\bigg[\Big | \frac{\mathcal{P}_{1-\e}(\tfrac{1}{2} + \ii \tau)}{5 \log\log T} \Big |^{4\ell}\bigg] \ll \frac{(2\ell)! \big(\sum_{p \leq T} p^{-1}\big)^{2\ell}}{5^{4\ell} (\log \log T)^{4\ell}} \ll \sqrt{\ell} \Big(\frac{4 \cdot 2}{25 e}\Big)^{2\ell} \ll e^{-4\ell} \ll (\log T)^{-8}.
		\end{equation}
        The estimate \eqref{eq:le:3rd.eq.1} then follows by applying the Cauchy-Schwarz inequality, the fourth moment bound $\E[|\zeta(\tfrac 12 + \ii \tau)|^4] \ll (\log T)^4$, see e.g.\ \cite{MR1575391}, and \eqref{eq:le:3rd.eq.1.in.proof}.
        For \eqref{eq:le:3rd.eq.2}, the same reasoning as in \eqref{eq:le:3rd.eq.1.in.proof} yields the estimate $\ll e^{-2\ell} \ll (\log T)^{-4}$.
	\end{proof}	

   The last three lemmas show a moment bound of the right order for  $\zeta\cdot e^{-\mathcal P_{|\theta|}}$.
    \begin{proposition}\label{pr:momenttwi}
        Let $-1 < \theta < 0$, $0 < \beta \leq 2$ and $\e > 0$ be given. Then, as $T \rightarrow \infty$,
        \begin{equation}
            \E\Big[\big|(\zeta \cdot e^{-\mathcal{P}_{|\theta|}})(\tfrac 12 + \ii \tau)\big|^{\beta} \bb{1}_{\mathcal{A}(T)}\Big] \ll (\log T)^{\beta^2 (1 + \theta) / 4 + \e},
        \end{equation}
        with the event
        \begin{equation}
            \mathcal{A}(T) =\big\{|\mathcal{P}_{|\theta|}(\tfrac{1}{2} + \ii \tau)| \leq 2 \log\log T\big\}. \vspace{2mm}
        \end{equation}
    \end{proposition}

    \begin{proof}
    	Let $0 < \beta < 2$. By Young's inequality with $p = 2 / \beta$ and $q = 2 / (2 - \beta)$,
    	\begin{equation}\label{eq:pr:momenttwi.beginning}
            \begin{aligned}
                \big|\zeta(\tfrac 12 + \ii \tau)\big|^{\beta}
                &\leq \tfrac{1}{p} \cdot |\zeta(\tfrac 12 + \ii \tau)|^2 \cdot e^{- \frac{2}{q} \Re \mathcal{P}_{1-\e}(\tfrac{1}{2} + \ii \tau)} + \tfrac{1}{q} \cdot e^{\frac{2}{p} \Re \mathcal{P}_{1-\e}(\tfrac{1}{2} + \ii \tau)}
        		\\[1mm]
                &= \tfrac{\beta}{2} \cdot |\zeta(\tfrac 12 + \ii \tau)|^2 \cdot e^{- (2 - \beta) \Re \mathcal{P}_{1-\e}(\tfrac{1}{2} + \ii \tau)} + \tfrac{2 - \beta}{2} \cdot e^{\beta \, \Re \mathcal{P}_{1-\e}(\tfrac{1}{2} + \ii \tau)}.
            \end{aligned}
    	\end{equation}
    	Note that \eqref{eq:pr:momenttwi.beginning} holds trivially for $\beta = 2$.
    	Hence, for $0 < \beta \leq 2$,
    	\begin{equation}\label{eq:pr:momenttwi.next}
            \begin{aligned}
    		\big|(\zeta \cdot e^{-\mathcal{P}_{|\theta|}})(\tfrac 12+ \ii \tau)\big|^{\beta}
            &\leq \tfrac \beta2 \, |\zeta(\tfrac 12 + \ii \tau)|^2 \cdot e^{- (2 - \beta) \Re \mathcal{P}_{1-\e}(\tfrac{1}{2} + \ii \tau) - \beta \, \Re \mathcal{P}_{|\theta|}(\tfrac{1}{2} + \ii \tau)} \\
            &\quad+ \tfrac{2 - \beta}{2} \, e^{\beta \, \Re \mathcal{P}_{1-\e}(\tfrac{1}{2} + \ii \tau) - \beta \, \Re \mathcal{P}_{|\theta|}(\tfrac{1}{2} + \ii \tau)}.
    	   \end{aligned}
        \end{equation}
    	
        On the event $\mathcal{A}(T) \cap \{|\mathcal{P}_{1-\e}(\tfrac{1}{2} + \ii \tau)| \leq 5 \log\log T\}$, we get, by the truncation \eqref{eqn: truncation} with $\nu = 100 \lfloor \log\log T \rfloor$ and the identity $|z + w|^2 \leq 2 (|z|^2 + |w|^2)$, that
    	\begin{equation}\label{eq:pr:momenttwi.eq.1}
            \begin{aligned}
                &e^{- (2 - \beta) \Re \mathcal{P}_{1-\e}(\tfrac{1}{2} + \ii \tau) - \beta \, \Re \mathcal{P}_{|\theta|}(\tfrac{1}{2} + \ii \tau)} \\[-2mm]
                &\quad\ll \bigg| \sum_{\substack{\Omega(n) \leq 100 \lfloor \log\log T \rfloor \\ p | n \implies \log p \leq \log^{1 - \e} T}} \frac{\mathfrak{F}_{-1, \beta / 2 - 1, \theta}(n)\mathfrak{g}(n)}{n^{1/2 + \ii \tau}} \bigg|^2 + (\log T)^{-200},
            \end{aligned}
    	\end{equation}
        where $\mathfrak{F}_{\alpha,\beta,\theta}(n)$ is the completely multiplicative function defined in \eqref{eq:multiplicative.function}.
    	Likewise, on the  same event, we have
    	\begin{equation}\label{eq:pr:momenttwi.eq.2}
            \begin{aligned}
                &e^{\beta \, \Re \mathcal{P}_{1-\e}(\tfrac{1}{2} + \ii \tau) - \beta \, \Re \mathcal{P}_{|\theta|}(\tfrac{1}{2} + \ii \tau)} \\[-2mm]
                &\quad\ll \bigg| \sum_{\substack{\Omega(n) \leq 100 \lfloor \log\log T \rfloor \\ p | n \implies \log p \leq \log^{1 - \e} T}} \frac{\mathfrak{F}_{0, \beta / 2, \theta}(n)\mathfrak{g}(n)}{n^{1/2 + \ii \tau}} \bigg|^2 + (\log T)^{-200}.
            \end{aligned}
    	\end{equation}
    	Finally, on the event $\mathcal{A}(T) \cap \{|\mathcal{P}_{1-\e}(\tfrac{1}{2} + \ii \tau)| > 5 \log\log T\}$, we get, for any $\ell \geq 1$,
    	\begin{equation}\label{eq:pr:momenttwi.eq.3}
            \big|(\zeta \cdot e^{-\mathcal{P}_{|\theta|}})(\tfrac 12+ \ii \tau)\big|^{\beta} \leq (\log T)^4 \cdot (1 + |\zeta(\tfrac 12 + \ii \tau)|^2) \cdot \Big | \frac{\mathcal{P}_{1-\e}(\tfrac{1}{2} + \ii \tau)}{5 \log\log T} \Big |^{2\ell},
    	\end{equation}
        since for $\beta \leq 2$, $|\zeta|^\beta$ is bounded by $(1+|\zeta|^2)$ and $|e^{- \mathcal P_{|\theta|}}|^{\beta}$ is bounded by $(\log T)^4$ on $\mathcal{A}(T)$.
        We choose $\ell = 2 \lfloor \log\log T \rfloor$. Now, take the expectation with $\tau$ restricted to $\mathcal{A}(T)$ in \eqref{eq:pr:momenttwi.next}, then split the terms on the right-hand side over the associated events in \eqref{eq:pr:momenttwi.eq.1}, \eqref{eq:pr:momenttwi.eq.2} and \eqref{eq:pr:momenttwi.eq.3}. We use Lemmas~\ref{le:1nd},~\ref{le:2nd} and~\ref{le:3rd} to bound the expectations.
    \end{proof}

    \subsection{Discretization}

    The analysis of the maximum of zeta on an interval can often be restricted to $h$'s on a grid with mesh of order $(\log T)^{-1}$.
    This can be proved for the maximum using the functional equation for zeta, see for example Lemma 2.2 in \cite{FGH07}.
    We will need a more elaborate variant for general Dirichlet polynomials.

	\begin{proposition}\label{prop:union}
		\hspace{-3mm}Let $\theta > - 1$, $\beta \geq 1$ and $\e > 0$ be given.
        Let $D(s) = \sum_{n\leq T^{1+\e}} a(n)n^{-s}$ be a Dirichlet polynomial of length $T^{1 + \e}$ where $\sup_{n\leq T^{1+\e}} |a(n)| \leq B$ for some $B > 0$ possibly depending on $\e$ and $T$.
		Then, for all $A > 10 (1 + \e) \beta$, $T \leq t \leq 2T$, and $\sigma\geq 1/2$,
		\begin{equation}\label{eq:le:union}
            \begin{aligned}
                &\sup_{|h| \leq \log^\theta T} |D(\sigma + \ii t + \ii h)|^{\beta} \\[-2mm]
                &\quad\qquad\ll_A \sum_{|k| \leq 2 \log^{1+\theta} T} \big | D \big ( \sigma + \ii t + \tfrac{2\pi \ii k}{(1 + 2\e) \log T} \big ) \big |^{\beta} \\[-1.5mm]
                &\quad\qquad\quad+ \sum_{2\log^{1+\theta} T < |k| \leq T} \big | D \big ( \sigma + \ii t + \tfrac{2\pi \ii k }{(1 + 2\e) \log T} \big ) \big |^{\beta} \cdot \frac{1}{1 + |k|^{A}} + B^{\beta} \, T^{-A/2}.
            \end{aligned}
		\end{equation}
	\end{proposition}	
	
	\begin{proof}
        Let $V$ be a smooth function with $V(x) = 1$ for $x\in [-(1+\e),0]$ and compactly supported in $[-(1+2\e), \e ]$. We show
		\begin{equation}\label{eq:discr}
            D(\sigma + \ii t + \ii h) = \frac{1}{1 + 2\e} \sum_{k \in \mathbb{Z}} D \big ( \sigma + \ii t + \tfrac{2\pi \ii k}{(1 + 2\e) \log T} \big ) \widehat{V} \big ( \tfrac{k}{1 + 2\e} - \tfrac{h \log T}{2\pi} \big ).
    		\end{equation}
    	By taking the complex norm and applying H\"older's inequality with $\beta \geq 1$, this yields
    	\begin{equation}
            \begin{aligned}
                |D(\sigma + \ii t + \ii h)|
                &\leq \Big(\frac{1}{1 + 2\e} \sum_{k \in \mathbb{Z}} \big | D \big ( \sigma + \ii t + \tfrac{2\pi \ii k}{(1 + 2\e) \log T} \big ) \big |^{\beta} \cdot \big | \widehat{V} \big ( \tfrac{k}{1 + 2\e} - \tfrac{h \log T}{2\pi} \big )\big | \Big)^{1/\beta} \\
                &\quad\times \Big( \frac{1}{1 + 2\e} \sum_{k \in \mathbb{Z}} \big | \widehat{V} \big( \tfrac{k}{1 + 2\e} - \tfrac{h \log T}{2\pi} \big) \big | \Big)^{1 - 1/\beta}.
            \end{aligned}
    	\end{equation}
        This proves \eqref{eq:le:union} after taking the supremum over $h$, using the rapid decay of $\widehat{V}$, and noticing that $\sup_{n\leq T^{1+\e}} |a(n)| \leq B$ and our assumption on $A$ imply that
        \begin{equation}
            \begin{aligned}
                &\sum_{|k| > T} \big | D \big ( \sigma + \ii t + \tfrac{2\pi \ii k }{(1 + 2\e) \log T} \big ) \big |^{\beta} \cdot \frac{1}{1 + |k|^{A}} \\[-2mm]
                &\qquad\ll B^{\beta} \, T^{(1 + \e) \beta / 2} \sum_{|k| > T} \frac{1}{1 + |k|^{A}} \ll B^{\beta} \, T^{-A/2}.
            \end{aligned}
        \end{equation}

        Since $D(s)$ is of the form $\sum_{n\leq T^{1+\e}}a(n)n^{-s}$, it suffices by linearity to establish \eqref{eq:discr} for a single $n\leq T^{1+\e}$, i.e.,
        \begin{equation}\label{eqn: discretize to prove}
            n^{-\ii h} = \frac{1}{1 + 2\e} \sum_{k \in \mathbb{Z}} n^{ - \tfrac{2\pi \ii k}{(1 + 2\e) \log T}} \cdot \widehat{V} \big ( \tfrac{k}{1 + 2\e} - \tfrac{h \log T}{2\pi} \big ), \quad 1\leq n\leq T^{1+\e}.
        \end{equation}
        Using the Poisson summation formula, the right-hand side can be rewritten as
		\begin{equation}
    		\begin{aligned}
                & \sum_{\ell \in \mathbb{Z}} \int_{\mathbb{R}} \frac{1}{1+2\e} \exp\Big(2\pi \ii y(-\ell-\tfrac{1}{1+2\e}\tfrac{\log n}{\log T})\Big)\cdot \widehat{V} \big ( \tfrac{y}{1 + 2\e} - \tfrac{h \log T}{2\pi} )\rd y\\
                &=  n^{-\ii h} \sum_{\ell \in \mathbb{Z}} e^{-\ii \ell (1+2\e)h\log T }\int_{\mathbb{R}}  \exp\Big(2\pi \ii u(-\ell (1+2\e)-\tfrac{\log n}{\log T})\Big)\cdot \widehat{V} \big ( u)\rd u\\
                &=  n^{-\ii h} \sum_{\ell \in \mathbb{Z}} e^{-\ii \ell (1+2\e)h\log T }\cdot V(-\ell (1+2\e)-\tfrac{\log n}{\log T}),\\
            \end{aligned}
		\end{equation}
        where we made the change of variable $y=(1+2\e)(u + \tfrac{h \log T}{2\pi})$.
        The term $\ell = 0$ is equal to $n^{-\ii h}$ since $V(-\tfrac{\log n}{\log T}) = 1$ for $1 \leq n\leq T^{1+\e}$ by the choice of $V$. The other terms ($\ell\neq 0$) are all equal to $0$ since $-\ell (1+2\e)-\tfrac{\log n}{\log T}$ falls outside the support of $V$ for $1\leq n\leq T^{1+\e}$. This proves \eqref{eqn: discretize to prove} and the proposition.
    \end{proof}

    Proposition~\ref{prop:union} implies five important corollaries to tackle the maximum of $\zeta$ and of Dirichlet polynomials.
	We first observe that the discretization applies to $\zeta$ in Corollary~\ref{cor:discrzeta}. This is a consequence of the following approximation.

    \begin{lemma}[Approximation of $\zeta$]\label{zetapprox}
        Let $\e > 0$ and $\sigma \geq 1/2$ be given, and let $k > \max\{5, 10 / \e\}$ be an integer.
        Then, as $T\to \infty$ and for $t\asymp T$, we have
        \begin{equation} \label{eq:zetapprox}
            \zeta(\sigma + \ii t) =
            \begin{cases}
                \sum_{n \leq T^{1 + \e}} n^{-\sigma - \ii t} w_k \big ( \log \frac{e^k n}{ T^{1 + \e}} \big ) + \OO_k(T^{-k \e / 2}), &\text{if } 1/2 \leq \sigma \leq 2, \\
                \sum_{n \leq T}  n^{-\sigma - \ii t} + \OO(T^{-1}), &\text{if } \sigma > 2,
            \end{cases}
        \end{equation}
        where the smoothing $w_{k}$ is defined by setting
        \begin{equation}
            w_k(x) \leqdef
            \begin{cases}
                1, &\text{if } x < 0, \\
                (-1)^k \sum_{\ell = 0}^k \binom{k}{\ell} \frac{(-1)^{\ell}}{k!} ( \ell - x )_+^k, &\text{if } 0 \leq x < k, \\
                0, &\text{if } x \geq k,
            \end{cases}
        \end{equation}
        where $(y)_+ \leqdef \max\{y,0\}$.
        Examples of graphs for $w_k(x)$ are provided in Figure~\ref{fig:example}.
    \end{lemma}

    \begin{figure}%
        \centering
        \subfloat[][\centering Graph of $w_3(x)$]{\includegraphics[width=5.9cm]{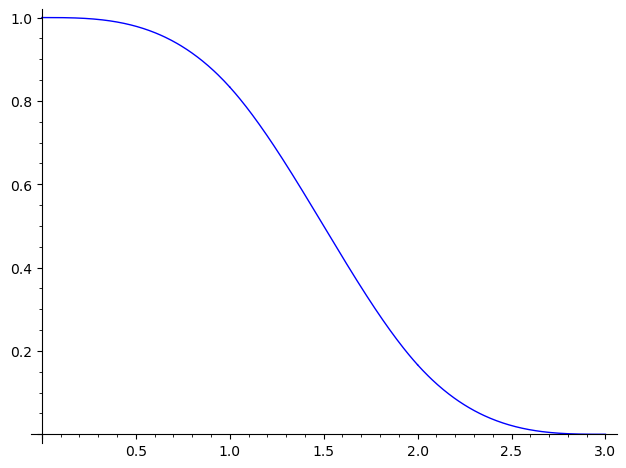}}%
        \qquad
        \subfloat[][\centering Graph of $w_6(x)$]{\includegraphics[width=5.9cm]{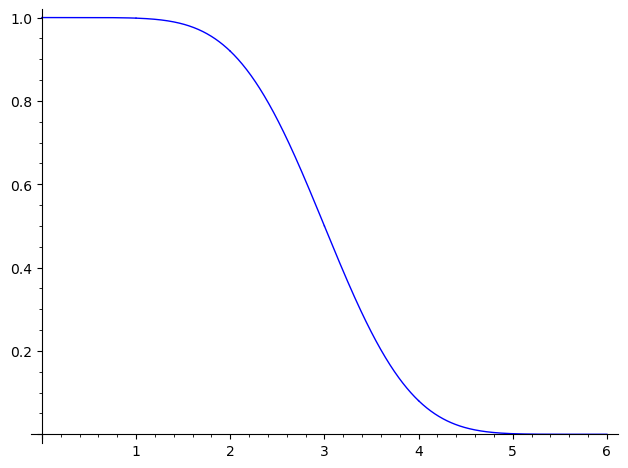}}%
        \caption{Examples of graphs for $w_k(x)$.}%
        \label{fig:example}%
    \end{figure}

    \begin{proof}
        The case $\sigma > 2$ is a trivial consequence of the fact that $|\sum_{n > T}  n^{-\sigma - \ii t}| \leq |\sum_{n > T}  n^{-2}| \leq T^{-1}$.
        Therefore, assume $1/2 \leq \sigma \leq 2$.
        We claim that,
        \vspace{-1mm}
        \begin{equation}
            w_{k}(x) = \frac{1}{2\pi \ii} \int_{2 - \ii \infty}^{2 + \ii \infty} e^{-x z} \Big ( \frac{e^z - 1}{z} \Big )^{k} \frac{\rd z}{z}.
        \end{equation}
        First it is easy to check that this formula holds for $x < 0 $ and $x > k$: if $x < 0$ then we shift the contour towards $\Re z = - \infty$ and collect a single pole with residue $1$ at $z = 0$, while if $x > k$ then we shift the contour towards $\Re z = \infty$ and we see that the integral is zero. In the remaining intermediate range $0 \leq x < k$ we expand
        \begin{equation}
            \Big ( \frac{e^{z} - 1}{z} \Big )^{k} = \frac{(-1)^k}{z^k} \sum_{\ell = 0}^{k} \binom{k}{\ell} (-1)^{\ell} e^{\ell z}
        \end{equation}
        and we use the fact that,
        \begin{equation}
            \frac{1}{2\pi \ii} \int_{2 - \ii \infty}^{2 + \ii \infty} e^{-x z} e^{\ell z} \frac{\rd z}{z^{k + 1}} = \frac{1}{k!}
            \begin{cases}
                (\ell - x)^k, & \text{ if } \ell - x > 0, \\
                0, & \text{ if } \ell - x \leq 0.
            \end{cases}
        \end{equation}

        Therefore,
        \begin{equation}
            \sum_{n \geq 1} \frac{1}{n^{\sigma + \ii t}} w_k \Big ( \log \frac{n}{T^{1 + \e}} \Big ) = \frac{1}{2\pi \ii} \int_{2 - \ii \infty}^{2 + \ii \infty} \zeta(\sigma + \ii t + z) T^{(1 + \e)z} \Big ( \frac{e^z - 1}{z} \Big )^{k} \frac{\rd z}{z}.
        \end{equation}
        We now shift the contour to the line $\Re z = - (k - 2)$. We collect a pole at $z = 0$ with residue $\zeta(\sigma + \ii t)$. On the line $\Re z = - (k - 2)$, we bound the integral using the estimate $|\zeta(r + \ii t)| \ll (1 + |t|)^{1/2 - r}$, which is valid for any fixed $r < - \tfrac{1}{100}$ and all $t \in \mathbb{R}$.%
        \footnote{This estimate follows from applying the functional equation for $\zeta(r + \ii t)$, bounding the ratio of Gamma factors using Stirling's formula and bounding $\zeta(1 - r - \ii t)$ trivially by $O(1)$.}
        Specifically, the contribution of the line $\Re z = - (k - 2)$ is bounded by
        \begin{equation}
            T^{- (1 + \e)(k - 2)} \int_{\mathbb{R}} (T^{k - 2} + |u|^{k - 2}) \frac{2^k \rd u}{| 2 - k + \ii u|^{k + 1}} \ll_{k} T^{-\e k / 2}.
        \end{equation}
        This proves that
        \begin{equation}
            \zeta(\sigma + \ii t) = \sum_{n\geq 1} \frac{1}{n^{\sigma + \ii t}} w_k \Big ( \log \frac{n}{ T^{1 + \e}} \Big ) + \OO_k(T^{-k \e / 2}), \quad \text{for } t\asymp T.
        \end{equation}
        The conclusion follows by a simple rescaling.
    \end{proof}

    From Lemma~\ref{zetapprox}, we derive the following discretization result.

    \begin{corollary}\label{cor:discrzeta}
        Let $\theta > - 1$, $\beta \geq 1$ and $\e > 0$ be given.
        For any $A > 10 (\e^{-1} + 1) \beta$ and all $T \leq t \leq 2T$,
        \begin{equation}
            \begin{aligned}
                &\max_{|h| \leq \log^{\theta}T} |\zeta(\tfrac 12 + \ii t + \ii h)|^{\beta} \\[-0.5mm]
                &\qquad\ll_A \sum_{|\ell| \leq 2\log^{1 + \theta} T} \big | \zeta \big (\tfrac{1}{2} + \ii t + \tfrac{2\pi \ii \ell}{(1 + 2\e) \log T} \big ) \big |^{\beta} \\[-1.5mm]
                &\qquad\quad+ \sum_{2 \log^{1 + \theta}T < |\ell| \leq T} \big | \zeta \big ( \tfrac{1}{2} + \ii t + \tfrac{2\pi \ii \ell}{(1 + 2\e) \log T} \big ) \big |^{\beta} \cdot \frac{1}{1 + |\ell|^A} + T^{-A/2}.
            \end{aligned}
        \end{equation}
	\end{corollary}

	\begin{proof}
        This is a consequence of the $\zeta$ approximation in Lemma~\ref{zetapprox} with $\sigma = 1/2$ and $k = \lfloor (A + 1)/ (\beta \e) \rfloor $, and the discretization in Proposition~\ref{prop:union} with $B = 1$.
	\end{proof}

    As a consequence, we get a suboptimal upper bound for $\theta > -1$ using the second moment. Note that this bound also works for $\theta$ dependent on $T$.

    \begin{corollary}\label{cor:maxbound}
        Let $0 < \e \leq 1$ be given and let $k > 10/\e$ be an integer.
        Then, for any $\theta > -1$, possibly dependent on $T$, we have
		\begin{equation}\label{eq:cor:maxbound.DP}
            \P\bigg(\max_{|h| \leq \log^\theta T} \bigg|\sum_{n \leq T^{1 + \e}} \frac{1}{n^{\frac{1}{2} + \ii \tau + \ii h}} w_k \Big ( \log \frac{e^k n}{ T^{1 + \e}} \Big )\bigg| > 2^{\theta} (\log T)^{2 + \theta}\bigg) \ll \frac{2^{-\theta}}{\log T}.
		\end{equation}
	\end{corollary}

	\begin{proof}
        The Dirichlet polynomial in \eqref{eq:cor:maxbound.DP} is $\ll \sum_{n \leq T^{1+\e}} n^{-1/2} \ll T^{(1+\e)/2} \ll T$, so the probability is just zero when $\theta > \log T / \log\log T$.
        Therefore, we assume that $\theta \leq \log T / \log\log T$. By the $\zeta$ approximation in Lemma~\ref{zetapprox}, it suffices to prove
        \begin{equation}\label{eq:cor:maxbound}
            \P\Big(\max_{|h| \leq \log^\theta T} |\zeta(\tfrac 12 + \ii \tau + \ii h)| > 2^{\theta/2} (\log T)^{2 + \theta}\Big) \ll \frac{2^{-\theta}}{\log T}.
        \end{equation}
        By applying Markov's inequality and Corollary~\ref{cor:discrzeta} with $A = 100 (\e^{-1} + 1)$, the probability in \eqref{eq:cor:maxbound} is
		\begin{equation}\label{eq:cor:maxbound.beginning}
            \begin{aligned}
        		&\leq 2^{-\theta} (\log T)^{-4 - 2\theta} \, \E\Big[\max_{|h| \leq \log^{\theta}T} |\zeta(\tfrac 12 + \ii \tau + \ii h)|^2\Big] \\[2mm]
                &\ll 2^{-\theta} (\log T)^{-4 -2\theta} \sum_{|\ell| \leq 2 \log^{1 + \theta}T} \E\Big[\big | \zeta \big (\tfrac 12 + \ii \tau + \tfrac{2\pi \ii \ell}{(1 + 2\e) \log T} \big ) \big |^{2}\Big] \\
                &\quad+ 2^{-\theta} (\log T)^{-4 - 2\theta} \sum_{2 \log^{1 + \theta}T < |\ell| \leq T} \E\Big[\big | \zeta \big ( \tfrac 12 + \ii \tau + \tfrac{2\pi \ii \ell}{(1 + 2\e) \log T} \big ) \big |^{2}\Big] \cdot \frac{1}{1 + |\ell|^{100}} + T^{-50}.
            \end{aligned}
		\end{equation}
        Using a standard second moment bound, see e.g.\ \cite[p.141]{MR882550}, the last two expectations are $\ll \log T$.
        We conclude that the right-hand side of \eqref{eq:cor:maxbound.beginning} is
		\begin{equation}
            \ll 2^{-\theta} (\log T)^{-2 - \theta} \ll \frac{2^{-\theta}}{\log T},
		\end{equation}
		since $\theta > -1$.
	\end{proof}

    A similar reasoning using Markov's inequality can be applied to get an upper bound for the maximum of $\mathcal P_\alpha$, $0<\alpha<1$.
    The bound below is suboptimal for $\theta < 0$ and optimal for $\theta \geq 0$.
    \begin{corollary}
        Let $\theta > -1$, $\e > 0$ and $\sigma \geq 1/2$ be given.
        Then,
        \begin{equation}\label{eqn: max DP}
            \P\Big(\max_{|h| \leq \log^\theta T} |\mathcal{P}_{\alpha}(\sigma + \ii \tau + \ii h)| > (\sqrt{\alpha(1+\theta)}+\e)\log\log T\Big) = \oo(1).
        \end{equation}
    \end{corollary}

    \begin{proof}
        We apply Markov's inequality with exponent $2\ell$, and discretize as in \eqref{eq:cor:maxbound.beginning} using Proposition~\ref{prop:union} with $B  = 1$. We then use moment estimates from Lemma~\ref{lem:moment.estimates.Lemma.3.Sound.2009.generalisation}, with $\ell = \lfloor (1+\theta)\log\log T \rfloor$, to bound the expectations.
    \end{proof}

    When $\theta<0$ and $\alpha>|\theta|$, the bound \eqref{eqn: max DP} (and its analogue for $\zeta$) needs to be refined by discarding the contribution of small primes. The result below directly implies that for $\theta<0$ and $\alpha>|\theta|$, the sharp upper bound for $\Re \mathcal{P}_{\alpha}$ is $\sqrt{(\alpha+\theta)(1+\theta)}\log\log T$ since the effective variance is $\tfrac{(\alpha+\theta)}{2}\log\log T$.

    \begin{corollary}\label{cor:DP.cutting.0.to.theta}
        Let $-1 < \theta < 0$ and $\sigma\geq 1/2$ be given.
        Then, for any $0 < \e < C$ and $V = V(T)$ that satisfies $\e \log \log T \leq V \leq C \log \log T$, we have
        \begin{equation}\label{eq:cor:DP.cutting.0.to.theta}
            \P\Big(\max_{|h| \leq \log^\theta T} \big|\mathcal{P}_{|\theta|}(\sigma + \ii \tau + \ii h)\big| > V\Big) \ll e^{-c V},
        \end{equation}
        for some constant $c = c(\e,C) > 0$.
    \end{corollary}

    \begin{proof}
        For a lighter notation, write $S(h) =\mathcal{P}_{|\theta|}(\sigma + \ii \tau + \ii h)$.
        (We keep the dependence on $\tau$ implicit, consistent with the probabilistic notation for random variables.)
        We have
        \vspace{-2mm}
        \begin{equation}\label{eq:cor:DP.cutting.0.to.theta.eq.first}
            \begin{aligned}
                \P\Big(\max_{|h| \leq \log^\theta T} |S(h)| > V\Big)
                &\leq \P\Big(\max_{|h| \leq \log^\theta T} |S(h) - S(0)| > V/2\Big) \\
                &\quad+ \P\big(|S(0)| > V/2\big).
            \end{aligned}
        \end{equation}
        Let $\ell$ denote a generic natural integer.
        By Markov's inequality, a moment estimate (Lemma~\ref{lem:moment.estimates.Lemma.3.Sound.2009.generalisation}) and a prime number theorem estimate (Lemma~\ref{lem:PNT.estimates}), we have
        \begin{equation}
            \P\big(|S(0)| > V/2\big) \leq \frac{\E\big[|S(0)|^{2\ell}\big]}{(V/2)^{2\ell}} \ll \ell \hspace{0.3mm}! \, \bigg(\frac{\sum_{p\leq T} p^{-2\sigma}}{(V/2)^2}\bigg)^{\ell} \ll \bigg(\frac{4 \ell \log \log T}{\e^2 (\log \log T)^2}\bigg)^{\ell}.
        \end{equation}
        With the choice $\ell = \lfloor \frac{\e^2}{8} \log \log T\rfloor$, this probability is $\ll \exp(-a V)$ for some constant $a = a(\e,C) > 0$.

        It remains to control the first probability on the right-hand side of \eqref{eq:cor:DP.cutting.0.to.theta.eq.first}.
        Let $\ell$ denote another natural integer to be chosen later.
        By applying Proposition~\ref{prop:union}, we get
        \begin{equation}\label{eq:cor:DP.cutting.0.to.theta.eq.Sobolev}
            \begin{aligned}
                \E\Big[\max_{|h| \leq \log^\theta T} |S(h) - S(0)|^{2\ell}\Big]
                &\ll \log^{1+\theta} T \cdot \max_{|h| \leq \log^{\theta} T} \E\big[|S(h) - S(0)|^{2\ell}\big]  \\
            \end{aligned}
        \end{equation}
        A short calculation, using moment estimates (Lemma~\ref{lem:moment.estimates.Lemma.3.Sound.2009.generalisation}) followed by prime number theorem estimates (Lemma~\ref{lem:PNT.estimates}), yields
        \begin{equation}\label{eq:cor:DP.cutting.0.to.theta.eq.Q}
            \max_{|h| \leq \log^{\theta} T} \E\big[|S(h) - S(0)|^{2\ell}\big] \ll \ell \hspace{0.3mm}! \, \bigg(\sum_{\log p\leq \log^{|\theta|}T} \frac{2 - 2\cos(|h| \log p)}{p}\bigg)^{\ell}
            \ll (\ell \, d)^{\ell},
        \end{equation}
        for some constant $d > 0$
        (to obtain the last inequality, note that $|h| \cdot \log^{|\theta|}T \leq 1$).

        Then, by Markov's inequality and the choice $\ell = \lfloor \frac{\e^2}{8d} \log \log T\rfloor$, we deduce
        \begin{equation}
            \P\Big(\max_{|h| \leq \log^\theta T} |S(h) - S(0)| > V/2\Big) \ll \log^{1+\theta} T \cdot \bigg(\frac{4 \, \ell \, d}{V^2}\bigg)^{\ell} \ll e^{-b V},
        \end{equation}
        for some constant $b = b(\e,C) > 0$.
    \end{proof}

    As before, the maximum of $\zeta \cdot e^{-\mathcal{P}_{|\theta|}}$ can be discretized by truncating the exponential.
	\begin{corollary}\label{cor:union}
        Let $- 1 < \theta \leq 0$ and $\e > 0$ be given. Then, there exists a constant $C = C(\theta,\e) > 0$ such that the event
		\begin{equation}
            \begin{aligned}
                &\max_{|h| \leq \log^{\theta}T} |(\zeta \cdot e^{-\mathcal{P}_{|\theta|}})(\tfrac 12 + \ii \tau + \ii h)|^2 \\[-0.5mm]
                &\qquad\leq C \sum_{|\ell| \leq 2\log^{1 + \theta} T} \big | (\zeta \cdot e^{-\mathcal{P}_{|\theta|}}) \big (\tfrac{1}{2} + \ii \tau + \tfrac{2\pi \ii \ell}{(1 + 2\e) \log T} \big ) \big |^2 \\[-1.5mm]
                &\qquad\quad+ C \sum_{2 \log^{1 + \theta}T < |\ell| \leq T} \big | (\zeta \cdot e^{-\mathcal{P}_{|\theta|}}) \big ( \tfrac{1}{2} + \ii \tau + \tfrac{2\pi \ii \ell}{(1 + 2\e) \log T} \big ) \big |^2 \cdot \frac{1}{1 + |\ell|^{100}} + T^{-50}.
            \end{aligned}
	   \end{equation}
        has probability $1 - \oo(1)$.
	\end{corollary}
	
	\begin{proof}
		Define the event
        \begin{equation}\label{eq:A.tilde.original}
            \widetilde{\mathcal{A}}(T) =\Big\{\max_{|h| \leq \log^\theta T} \big|\mathcal{P}_{|\theta|}(\tfrac{1}{2} + \ii \tau + \ii h)\big| \leq 2 \log\log T\Big\}.
        \end{equation}
        By Corollary~\ref{cor:DP.cutting.0.to.theta}, we have $\P(\widetilde{\mathcal{A}}(T)) = 1 - \oo(1)$.
		By \eqref{eqn: truncation}, for all $\tau \in \widetilde{\mathcal{A}}(T)$, we also have
    	\begin{equation}\label{eq:cor:union.eq.2}
            \begin{aligned}
                \Big|\sum_{\substack{\Omega (n) \leq 20 \lfloor \log\log T \rfloor \\ p | n \implies \log p \leq \log^{|\theta|} T}} \frac{(-1)^{\Omega(n)}\mathfrak{g}(n)}{n^{1/2 + \ii \tau + \ii h}}\Big|
                &= \Big|e^{-\mathcal{P}_{|\theta|}(\tfrac{1}{2} + \ii \tau + \ii h)} \Big|  + \OO \big((\log T)^{-20}\big) \\[-5mm]
                &\asymp \Big|e^{-\mathcal{P}_{|\theta|}(\tfrac{1}{2} + \ii \tau + \ii h)} \Big|.
            \end{aligned}
    	\end{equation}
        Combining this with the $\zeta$ approximation in Lemma~\ref{zetapprox} with $\sigma = 1/2$ and $k = 102/\e$, we conclude that, for all $\tau \in \widetilde{\mathcal{A}}(T)$ and uniformly for $y\asymp T$,
		\begin{equation}\label{eq:cor:union.eq.3}
            \big|(\zeta \cdot e^{-\mathcal{P}_{|\theta|}})(\tfrac 12 + \ii \tau + \ii y)\big|  \asymp |D(\tfrac 12 + \ii \tau + \ii y)| + \OO\bigg(\frac{\log^2 T}{T^{51}}\bigg),
		\end{equation}
        where $D$ is a Dirichlet polynomial of length $T^{1 + 2\e}$.
        Proposition~\ref{prop:union} implies
		\begin{equation}
            \begin{aligned}
                &\max_{|h| \leq \log^{\theta}T} |D(\tfrac 12 + \ii \tau + \ii h)|^2 \\[-1mm]
                &\qquad\ll \sum_{|\ell| \leq 2\log^{1 + \theta} T} \big | D \big (\tfrac{1}{2} + \ii \tau + \tfrac{2\pi \ii \ell}{(1 + 2\e) \log T} \big ) \big |^2 \\[-1.5mm]
                &\qquad\quad+ \sum_{2 \log^{1 + \theta}T < |\ell| \leq T} \big | D \big ( \tfrac{1}{2} + \ii \tau + \tfrac{2\pi \ii \ell}{(1 + 2\e) \log T} \big ) \big |^2 \cdot \frac{1}{1 + |\ell|^{100}} + T^{-50}.
            \end{aligned}
	   \end{equation}
		Together with \eqref{eq:cor:union.eq.3}, this concludes the proof.
	\end{proof}

    \subsection{Proofs of the upper bounds}\label{se:upperbounds}

    \subsubsection{The case \texorpdfstring{$\theta \geq 0$}{theta bigger or equal to 0}}

    \begin{proof}[Proof of Theorem~\ref{thm: max} for $\theta \geq 0$]
    	By Markov's inequality with exponent $\beta > 0$, we have
    	\begin{equation}
            \begin{aligned}
        		&\P \Big ( \max_{|h| \leq \log^{\theta}T} |\zeta(\tfrac 12 + \ii \tau + \ii h)| > (\log T)^{m(\theta) + \e} \Big ) \\
                &\hspace{10mm}\ll (\log T)^{-\beta m(\theta) - \beta \e} \cdot \mathbb{E} \Big [ \max_{|h| \leq \log^{\theta}T} |\zeta(\tfrac 12 + \ii \tau + \ii h)|^{\beta} \Big ].
            \end{aligned}
    	\end{equation}
        If we choose $\beta = 2 m(\theta) \geq 2$, we get, by picking $A$ large enough in Corollary~\ref{cor:discrzeta}, that the right-hand side of the above equation is
    	\begin{equation}
            \ll (\log T)^{- 2 m(\theta)^2 - 2 \e + 1 + \theta} \cdot \mathbb{E} \Big [ |\zeta(\tfrac 12 + \ii \tau)|^{\beta} \Big ].
    	\end{equation}
        By applying Proposition~\ref{pr:maksbound} if $\beta \leq 4$ (i.e., if $\theta \leq 3$) and Proposition~\ref{pr:soundbound} if $\beta > 4$ (i.e., if $\theta > 3$), the expectation is bounded by $(\log T)^{m(\theta)^2 + \e}$.
    	Therefore, the claim follows.
    \end{proof}

    \begin{proof}[Proof of Theorem~\ref{thm: freezing} for $\theta \geq 0$]
        For all $\beta > 0$, Markov's inequality yields
        \begin{equation}\label{eq:proof.thm.1.1.thete.bigger.1.beginning}
            \begin{aligned}
                &\P \Big ( \int_{|h| \leq \log^\theta T} |\zeta(\tfrac 12 + \ii \tau + \ii h)|^{\beta} \rd t \geq (\log T)^{f_{\theta}(\beta) + \e} \Big ) \\
                &\hspace{10mm}\ll (\log T)^{-f_{\theta}(\beta) - \e} \log^\theta T \cdot \E\Big[|\zeta(\tfrac 12 + \ii \tau)|^{\beta}\Big].
            \end{aligned}
        \end{equation}
        When $\beta\leq 2\sqrt{1+\theta}$, we have $f_{\theta}(\beta) = \beta^2/4 + \theta$, so the right-hand side of \eqref{eq:proof.thm.1.1.thete.bigger.1.beginning} is $\ll (\log T)^{-\e / 2}$ by Proposition~\ref{pr:maksbound} for $\theta \leq 3$ and by Proposition~\ref{pr:soundbound} for $\theta > 3$.

        It remains to sharpen the bound in the case $\beta > 2 \sqrt{1 + \theta}$. We use the Lebesgue measure of high points.
        Let $a,b > 0$. Two successive applications of Markov's inequality yield
        \begin{equation}
            \begin{aligned}
                &\P\Big ( \m \Big \{ |h| \leq \log^\theta T : |\zeta(\tfrac 12 + \ii \tau + \ii h)| > (\log T)^{a} \Big \} \geq (\log T)^{- a^2 + \theta +  \e} \Big ) \\[1mm]
                &  \ll (\log T)^{a^2 - \e} \cdot (\log T)^{-b a} \cdot \mathbb{E} \Big [ |\zeta(\tfrac 12 + \ii \tau)|^{b} \Big ].
            \end{aligned}
        \end{equation}
        Again, the optimal bound is at $b=2a$. Using Proposition~\ref{pr:maksbound} for $\theta \leq 3$ and Proposition~\ref{pr:soundbound} for $\theta > 3$ and choosing $b = 2 a$, we conclude that this is $\ll (\log T)^{-\e / 2}$ for $0<a\leq m(\theta)$.

        We now partition the integral according to the value of the integrand. Let $M \geq 1$ be an integer and $0\leq j\leq M$. Theorem~\ref{thm: max} (for $\theta\geq 0$) and the above imply that, with probability $1-\oo(1)$,
        \begin{equation}
            \begin{aligned}
                &\int_{|h| \leq \log^\theta T} |\zeta(\tfrac 12 + \ii \tau + \ii h)|^{\beta} \rd h \\[1mm]
                &\quad\ll \sum_{0 \leq j \leq M} (\log T)^{\beta ((j + 1)/ M) m(\theta)} \cdot (\log T)^{-(j / M)^2 m(\theta)^2 + \theta + \e}.
            \end{aligned}
        \end{equation}
        For $\beta > 2 \sqrt{1 + \theta} \geq 2 m(\theta)$, the last term $j = M$ dominates and, in particular, the above is bounded by
        \begin{equation}
            \ll (\log T)^{\beta m(\theta) - m(\theta)^2 + \theta + 2 \e} =  (\log T)^{\beta m(\theta) - 1 + 2 \e},
        \end{equation}
        provided that $M$ is chosen sufficiently large with respect to $\theta$, $\beta$ and $\e$.
    \end{proof}

    \begin{remark}
        In the above proof, we could have handled all $\beta$'s using the Lebesgue measure of high points in the spirit of a Gibbs variational principle.
        We chose to prove the case $\beta\leq 2\sqrt{1+\theta}$ directly as the proof is straightforward.
    \end{remark}

    \subsubsection{The case \texorpdfstring{$\theta < 0$}{theta smaller than 0}}

    \begin{proof}[Proof of Theorem~\ref{thm: max} for \texorpdfstring{$\theta < 0$}{theta smaller than 0}]
        We notice that
        \begin{equation}\label{eq:restr1}
            \begin{aligned}
                &\P\Big ( \max_{|h| \leq \log^\theta T} |\zeta(\tfrac 12 + \ii \tau + \ii h)| > (\log T)^{m(\theta) + \e} \Big )
                \\[1mm]
                &\quad\leq \P \Big ( \max_{|h| \leq \log^\theta T} \big|(\zeta \cdot e^{-\mathcal{P}_{|\theta|}})(\tfrac 12 + \ii \tau + \ii h)\big| > (\log T)^{m(\theta) + \e / 2} \Big ) \\
                &\hspace{10mm}+ \P \Big ( \max_{|h| \leq \log^\theta T} \big|e^{\mathcal{P}_{|\theta|}(\tfrac{1}{2} + \ii \tau + \ii h)}\big| > (\log T)^{\e / 2} \Big ).
            \end{aligned}
        \end{equation}
        By Corollary~\ref{cor:DP.cutting.0.to.theta}, the last term is $\oo(1)$ as $T \rightarrow \infty$.
        As in \eqref{eq:A.tilde.original}, let
        \begin{equation}\label{eq:A.tilde}
            \widetilde{\mathcal{A}}(T) =\Big\{\max_{|h| \leq \log^\theta T} \big|\mathcal{P}_{|\theta|}(\tfrac{1}{2} + \ii \tau + \ii h)\big| \leq 2 \log\log T\Big\}.
        \end{equation}
        By Corollary~\ref{cor:DP.cutting.0.to.theta} again, the probability of $\widetilde{\mathcal{A}}(T)$ is $1 - \oo(1)$.
        We let $\mathcal{A}_0(T)$ denote the subset of $\widetilde{\mathcal{A}}(T)$ for which the conclusion of Corollary~\ref{cor:union} holds. The probability of $\mathcal{A}_0(T)$ is $1 - \oo(1)$.
    	Then, by Markov's inequality, we have
    	\begin{equation}
            \begin{aligned}
                &\P \Big(\Big\{\max_{|h| \leq \log^\theta T} \big|(\zeta \cdot e^{-\mathcal{P}_{|\theta|}})(\tfrac 12 + \ii \tau + \ii h)\big| > (\log T)^{m(\theta) + \e / 2}\Big\} \cap \mathcal{A}_0(T)\Big)  \\[1mm]
                &\quad\leq (\log T)^{- 2 m(\theta) - \e} \cdot \E\Big[\max_{|h| \leq \log^\theta T} \big|(\zeta \cdot e^{-\mathcal{P}_{|\theta|}})(\tfrac 12 + \ii \tau + \ii h)\big|^2 \, \bb{1}_{\mathcal{A}_0(T)}\Big].
            \end{aligned}
    	\end{equation}
    	By Corollary~\ref{cor:union}, and since $m(\theta) = 1 + \theta$, this is
    	\begin{equation}
            \ll (\log T)^{-(1 + \theta) - \e} \cdot \E\Big[\big|(\zeta \cdot e^{-\mathcal{P}_{|\theta|}})(\tfrac 12 + \ii \tau)\big|^2 \, \bb{1}_{\widetilde{\mathcal{A}}(T)}\Big].
    	\end{equation}
    	By Proposition~\ref{pr:momenttwi}, this is
    	\begin{equation}
    		\ll (\log T)^{-(1 + \theta) - \e} \cdot (\log T)^{(1 + \theta) + \e/2} \ll (\log T)^{-\e/2},
    	\end{equation}	
    	as needed.
    \end{proof}

    \begin{proof}[Proof of Theorem~\ref{thm: freezing} for $\theta < 0$]
        Similarly to \eqref{eq:restr1}, we can restrict the integrand to $\zeta \cdot e^{-\mathcal{P}_{|\theta|}}$ as follows
        \begin{equation}\label{eq:toap}
            \begin{aligned}
            &\P\Big ( \int_{|h| \leq \log^\theta T} |\zeta(\tfrac 12 + \ii \tau + \ii h)|^{\beta} \rd h > (\log T)^{f_{\theta}(\beta) + \e} \Big ) \\[1mm]
            &\quad\leq \P\Big(\int_{|h| \leq \log^\theta T} \big|(\zeta \cdot e^{-\mathcal{P}_{|\theta|}})(\tfrac 12 + \ii \tau + \ii h)\big|^{\beta} \rd h  > (\log T)^{f_{\theta}(\beta) + \e / 2}\Big) + \oo(1).
            \end{aligned}
        \end{equation}
        As in \eqref{eq:A.tilde}, the probability is $\P(\widetilde{\mathcal{A}}(T)) = 1 - \oo(1)$, and by Markov's inequality, we have
        \begin{equation}
            \begin{aligned}
                &\P\Big(\Big\{\int_{|h| \leq \log^\theta T} \big|(\zeta \cdot e^{-\mathcal{P}_{|\theta|}})(\tfrac 12 + \ii \tau + \ii h)\big|^{\beta} \rd h  > (\log T)^{f_{\theta}(\beta) + \e / 2}\Big\} \cap \widetilde{\mathcal{A}}(T)\Big) \\[1mm]
                &\quad\ll (\log T)^{-f_{\theta}(\beta) - \e / 2} \cdot \log^\theta T \cdot \E\Big[\big|(\zeta \cdot e^{-\mathcal{P}_{|\theta|}})(\tfrac 12 + \ii \tau)\big|^{\beta} \, \bb{1}_{\widetilde{\mathcal{A}}(T)}\Big].
            \end{aligned}
        \end{equation}
        By Proposition~\ref{pr:momenttwi}, the above is
        \begin{equation}
            \ll (\log T)^{-(\beta^2 / 4)( 1 + \theta) - \e / 2} \cdot (\log T)^{(\beta^2 / 4) \cdot (1 + \theta) + \e / 4} \ll (\log T)^{-\e / 4}.
        \end{equation}
        This bound proves the claim for $\beta\leq 2$.

        It remains to refine the bound for the case $\beta > 2$. This proceeds in the same way as in the proof of Theorem~\ref{thm: freezing} in the case $\theta \geq 0$, with $\zeta$ replaced by $\zeta \cdot e^{-\mathcal{P}_{|\theta|}}$ restricted on the event $\widetilde{\mathcal{A}}(T)$.
        Namely, we have, for $0 < a \leq m(\theta)$,
        \begin{equation}
            \begin{aligned}
                &\P\Big(\Big\{\m\big\{|h| \leq \log^\theta T : |(\zeta \cdot e^{-\mathcal{P}_{|\theta|}}) (\tfrac 12 + \ii \tau + \ii h)| > (\log T)^{a}\big\} \geq (\log T)^{- a^2 + \theta +  \e}\Big\} \cap \widetilde{\mathcal{A}}(T)\Big) \\
                &\ll (\log T)^{a^2 - \e} \cdot (\log T)^{-b a} \cdot \mathbb{E} \Big[ |(\zeta \cdot e^{-\mathcal{P}_{|\theta|}})(\tfrac 12 + \ii \tau)|^{b} \, \bb{1}_{\widetilde{\mathcal{A}}(T)} \Big].
            \end{aligned}
        \end{equation}
        This is $\oo(1)$ by Proposition~\ref{pr:momenttwi} with the optimal choice $b=2a/(1+\theta)\leq 2$.
        The remainder is done exactly as  in the proof of Theorem~\ref{thm: freezing} in the case $\theta \geq 0$, by partitioning the integral over values of the integrand in the range $[0,m(\theta)+\e]$.
    \end{proof}

\section{Lower bounds}

    In this section, we prove:
    \begin{proposition}\label{prop:lower.bound.maximum.zeta}
        Let $\theta > - 1$ and $\e > 0$ be given.
        Then,
        \begin{equation}\label{eq:cor:lower.bound.maximum.zeta}
            \P\Big(\max_{|h| \leq \log^{\theta}T} |\zeta(1/2 + \ii \tau + \ii h)| > (\log T)^{m(\theta) - \e}\Big)=1-\oo(1).
        \end{equation}
    \end{proposition}

    \begin{proposition}\label{prop:lower.bound.moments.zeta}
        Let $\theta > -1$, $\beta > 0$ and $\e > 0$ be given.
        Then,
        \begin{equation}\label{eq:prop:lower.bound.moments.zeta}
            \P\Big(\int_{-\log^{\theta}T}^{\log^{\theta}T} |\zeta(1/2 + \ii \tau + \ii h)|^{\beta} \rd h > (\log T)^{f_{\theta}(\beta) - \e}\Big)=1-\oo(1).
        \end{equation}
    \end{proposition}
    The lower bound for the maximum will be an easy consequence of the lower bound for the moments.
    The idea is to approximate zeta by an appropriate Dirichlet polynomial.
    This can be done with good precision off-axis, cf.\ Section~\ref{sect: off-axis}.
    The approximation to a Dirichlet polynomial is then shown in Section~\ref{sect: mollification}.
    The lower bound for the moments of the Dirichlet polynomials is proved in Section~\ref{sect: bounds D} using Kistler's multiscale second moment method.
    Finally, the two propositions above are proved in Section~\ref{sect: lower proofs}.

    \subsection{Reduction off-axis}\label{sect: off-axis}

    In \cite{ABBRS_2019}, the maximum on a short interval of the critical line was compared to the one on a short interval away from the critical line by exploiting the analyticity of $\zeta$ away from its pole. More precisely, a value off-axis can be seen as an average of zeta over the critical line weighed by the corresponding Poisson kernel.
    This approach could also be used in the case of the moments by using the subharmonicity of the function $z\mapsto |z|^\beta$.
    We choose to apply a different method based on the following convexity theorem of Gabriel, which handles error terms more efficiently.

	\begin{proposition}[Theorem 2 of \cite{gabriel_1927} in the special case $a = b = 1$]\label{prop:thm.gabriel.1927}
        Let $F$ be a complex valued function which is analytic in the strip $\alpha \leq \Re z \leq \beta$. Suppose that $|F(z)|$ tends to zero as $|\Im z| \rightarrow \infty$, uniformly for $\alpha \leq \Re z \leq \beta$. Then, for any $\gamma \in [\alpha, \beta]$ and any $p > 0$,
		\begin{equation}
			I(\gamma) \leq I(\alpha)^{(\beta - \gamma) / (\beta - \alpha)} \cdot I(\beta)^{(\gamma - \alpha) / (\beta - \alpha)},
		\end{equation}
		where
		\begin{equation}\label{eq:def.I.Gabriel}
            I(\sigma) \leqdef \int_{\mathbb{R}} |F(\sigma + \ii t)|^p \rd t.
		\end{equation}
	\end{proposition}
	
    This theorem has the following useful consequence.

    \begin{corollary}\label{cor:gabriel}
        Let $F$ be a complex valued function which is analytic in the strip $\tfrac 12 \leq \Re z $.
        Suppose that $|F(z)|$ tends to zero as $|\Im z| \rightarrow \infty$, uniformly for $\tfrac 12 \leq \Re z $.
    	Suppose also that $I(\sigma) \rightarrow 0$ as $\sigma \rightarrow \infty$. Then, for any $\sigma > \tfrac 12$ and any $p > 0$,
    	\begin{equation}
    		I(\sigma) \leq I(\tfrac 12).
    	\end{equation}
    \end{corollary}
	
    \begin{proof}
    	Let $\sigma^{\star}$ be such that
    	\begin{equation}\label{eq:def.sigma.star}
    		I(\sigma^{\star}) = \sup_{\sigma \geq 1/2} I(\sigma).
    	\end{equation}
        Note that because of the assumption that $I(\sigma) \rightarrow 0$ as $\sigma \rightarrow \infty$, the above $\sigma^{\star}$ has a finite value. Let $\e > 0$ be given. If $\sigma^{\star} = \tfrac 12$, then we are done. If $\sigma^{\star} \neq \tfrac 12$, then by Proposition~\ref{prop:thm.gabriel.1927} applied with $\gamma = \sigma^{\star}$, $\alpha = \frac{1}{2}$ and $\beta = \sigma^{\star} + \e$, we get
    	\begin{equation}
    		I(\sigma^{\star}) \leq I(\tfrac 12)^{\lambda} \cdot I(\sigma^{\star} + \e)^{\mu},
    	\end{equation}
        for some appropriate $\lambda, \mu > 0$ that satisfy $\lambda + \mu = 1$.

    	Therefore, by definition of $\sigma^{\star}$ in \eqref{eq:def.sigma.star},
    	\begin{equation}
    		I(\sigma^{\star}) \leq I(\tfrac 12)^{\lambda} \cdot I(\sigma^{\star})^{\mu},
    	\end{equation}
        and hence $I(\sigma^{\star})^{\lambda} \leq I(\tfrac 12)^{\lambda}$.
        Since $\lambda > 0$, we get $I(\sigma^{\star}) \leq I(\tfrac 12)$.
        The claim follows from \eqref{eq:def.sigma.star}.
    \end{proof}

    We now construct a special analytic approximation for the indicator function of the rectangle $\mathcal{R} = \{ \sigma + \ii v : \tfrac 12 \leq \sigma \leq \tfrac 12 + K, |v| \leq L\}$ for $K,L > 0$. The effective width of the indicator function will be $K\approx L/\Delta$ in the statement below.

	\begin{lemma}\label{le:constr}
        Let $b_1\in (0,1)$ and $\Delta, L, A, b_2 > 0$ be given. There exists an entire function $\Phi_{\Delta, L}(z)$ such that, for $z = \sigma + \ii v$ with $\sigma \geq \tfrac 12$ and $v \in \mathbb{R}$,
    	\begin{enumerate}[(i)]
    		\item For $|v| \geq (1 + b_2) L$, uniformly in $\sigma\geq \frac{1}{2}$,
    		$
    			\Phi_{\Delta, L} (z) \ll_{A} b_2^{-A} \Delta^{1-A}.
    		$\vspace{1mm}
    		\item For any $|v| \leq (1 - b_1) L$,
    		$
    			|\Phi_{\Delta, L}(z)| = 1 + \OO_{b_1,A}(\Delta^{-A}) + \OO( (\sigma - \tfrac 12) \tfrac{\Delta^2}{L} ).
    		$\vspace{1mm}
    		\item For any $|v| \leq (1 + b_2) L$,
    		$
    			|\Phi_{\Delta, L}(z)| \ll 1 + (\sigma - \tfrac 12) \tfrac{\Delta^2}{L}.
    		$\vspace{1.3mm}
    		\item
    		$
    			\Phi_{\Delta, L}(z) \rightarrow 0
    		$\vspace{1mm}
    			uniformly in $v$ as $\sigma \rightarrow \infty$.
    	\end{enumerate}
	\end{lemma}
	
	\begin{proof}    	
    	Let $V$ be a smooth function, compactly supported in $(0, \infty)$ and such that $V(1) = 1$.
    	Given a parameter $\eta > 0$ and given $z \in \mathbb{C}$ with $\Re z \geq \tfrac 12$ and $u \in \mathbb{R}$, consider the following function:
    	\begin{equation}
    		\delta_{\eta}(z) =  \int_{0}^{\infty} e^{- 2\pi (z - \frac{1}{2}) x} \cdot V \big ( \eta x \big ) \, \eta \rd x.
    	\end{equation}
        Then $\delta_{\eta}(z)$ defines an entire function of exponential type. By integration by parts, we see that
        \begin{equation}\label{eqn: recall}
            \delta_{\eta}(z) \ll_{A} (1 + |z - \tfrac 12| \eta^{-1})^{-A},
        \end{equation}
         for any $A > 0$ and uniformly in $\Re z \geq \tfrac 12$.  Therefore, we may think of $\delta_{\eta}(z)$ as localizing to $z = \tfrac 12 + \OO(\eta)$. Furthermore, notice that if $z = \tfrac 12 + \ii v$ and $u \in \mathbb{R}$, then
        \begin{equation}
        	\delta_{\eta}(z - \ii u) = \widehat{V}((v - u) \eta^{-1}),
        \end{equation}
        and for $z = \sigma + \ii v$, we have by a Taylor expansion of the exponential,
        \begin{equation}\label{eq:deltabound2}
            \begin{aligned}
                \delta_{\eta}(z  - \ii u)
                & = \int_{0}^{\infty} e^{- 2\pi (\sigma - \frac{1}{2} + \ii (v - u)) x} \cdot V \big ( \eta x \big ) \, \eta \rd x \\[0.5mm]
                & = \int_{0}^{\infty} e^{- 2\pi \ii (v - u) x} \cdot \Big ( 1 + \OO \big ( (\sigma - \tfrac 12) x\big ) \Big ) \cdot V \big ( \eta x \big ) \, \eta \rd x \\[1.5mm]
                & = \widehat{V}((v - u) \eta^{-1}) + \OO \big ( ( \sigma - \tfrac 12) \eta^{-1} \big ).
            \end{aligned}
        \end{equation}
        Finally, for $z = \sigma + \ii v$ with $\sigma \geq \tfrac{1}{2}$, we have from \eqref{eqn: recall} that
        \begin{equation}\label{eq:deltabound}
        	|\delta_{\eta}(z - \ii u)| \ll_{A} \frac{1}{1 + (|v - u| \eta^{-1})^A}.
        \end{equation}

        The candidate function is for $\eta=L/\Delta$,
        \begin{equation}\label{eq:candidate.Phi}
            \Phi_{\Delta, L}(z) =\frac{\Delta}{L} \int_{-L}^{L} e^{-2\pi \ii u \frac{\Delta}{L}} \cdot \delta_{L/\Delta} (z - \ii u) \rd u.
    	\end{equation}
    	
        We will now describe some of the features of this function. Write $z = \sigma + \ii v$ with $\sigma \geq \tfrac{1}{2}$. Using the bound \eqref{eq:deltabound}, we see that, if $|v| > (1 + b_2) L$ with $b_2 > 0$, then
    	\begin{equation}
            \Phi_{\Delta, L}(z) \ll_{A} \frac{\Delta}{L} \int_{-L}^{L} \frac{1}{1 + (|v - u| \, \tfrac{\Delta}{L} )^{A}} \, \rd u \ll_{A} b_2^{-A} \Delta^{1 - A}.
    	\end{equation}
        This gives the first claim.

        If $|v| \leq (1 - b_1) L$, then by \eqref{eq:candidate.Phi} and \eqref{eq:deltabound2}, we have
    	\begin{equation}
            \Phi_{\Delta, L}(z) = \frac{\Delta}{L} \int_{-L}^{L} e^{-2\pi \ii u \frac{\Delta}{L}} \cdot \widehat{V} \big ((v - u) \, \tfrac{\Delta}{L} \big ) \rd u + \OO \big ( (\sigma - \tfrac 12) \tfrac{\Delta^2}{L} \big ).
    	\end{equation}
        It follows that if $\tfrac 12 \leq \sigma $ and $|v| \leq (1 - b_1) L$, then due to the rapid decay of $\widehat{V}$, we have
    	\begin{equation}
            \begin{aligned}
                \Phi_{\Delta , L}(z)
                & = e^{-2\pi \ii v \frac{\Delta}{L}} \int_{v \frac{\Delta}{L} - \Delta}^{v \frac{\Delta}{L} + \Delta} e^{2\pi \ii u} \cdot \widehat{V}(u) \rd u +  \OO \big ( (\sigma - \tfrac 12) \tfrac{\Delta^2}{L} \big ) \\[1mm]
                & = e^{-2\pi \ii v \frac{\Delta}{L}} + \OO_{b_1,A}(\Delta^{-A}) + \OO \big ( (\sigma - \tfrac 12) \tfrac{\Delta^2}{L} \big ),
            \end{aligned}
    	\end{equation}
        by Fourier inversion and the assumption that $V(1) = 1$. This proves the second claim.
        If $\tfrac 12 \leq \sigma \ll 1$ and $|v| \leq (1 + b_2) L$, then we have the bound
    	\begin{equation}
    		|\Phi_{\Delta, L}(z)| \ll \int_{\mathbb{R}} |\widehat{V}(u)| \rd u + \OO \big ( ( \sigma - \tfrac 12) \tfrac{\Delta^2}{L} \big ),
        \end{equation}
        which proves the third claim.

        Finally, notice that $\delta_{L/\Delta}(z - \ii u) \rightarrow 0$ uniformly as $\sigma \rightarrow \infty$ by \eqref{eqn: recall}, which implies the last claim that $\Phi_{\Delta, L}(z) \rightarrow 0$ uniformly in $v \in \mathbb{R}$ as $\sigma \rightarrow \infty$. 	
	\end{proof}

	The following proposition relates the moments off and on axis.

	\begin{proposition}\label{prop: LB}
        Let $\theta > -1$, $\beta > 0$, $0 < \e \leq 1$ and $T \geq 10^9$ be given.
        Then, for all $\tfrac 12 \leq \sigma \leq \tfrac 12 + (\log T)^{\theta - 3 \e}$, the event
		\begin{equation}
		\label{eqn: int off axis}
            \int_{-\log^{\theta}T}^{\log^{\theta}T} |\zeta(\sigma + \ii \tau + \ii u)|^{\beta} \rd u \ll \int_{-3\log^{\theta}T}^{3\log^{\theta}T} |\zeta(\tfrac 12 + \ii \tau + \ii u)|^{\beta} \rd u + \frac{1}{(\log T)^{96}}
		\end{equation}
        has probability $1 - \oo(1)$.
	\end{proposition}

    \begin{proof}
    	Let
        \begin{equation}
            D(\sigma + \ii \tau) =
            \begin{cases}
                \sum_{n \leq T^{1 + \e}} n^{-\sigma - \ii \tau} w_k \Big ( \frac{e^k n}{T^{1 + \e}} \Big ), &\text{if } 1/2 \leq \sigma \leq 2, \\
                \sum_{n \leq T}  n^{-\sigma - \ii \tau}, &\text{if } \sigma > 2.
            \end{cases}
        \end{equation}
        with $0 < \e \leq 1$ and $k > 10 / \e$ a fixed integer.
        Using the $\zeta$ approximation in Lemma~\ref{zetapprox}, we have, for $T \leq \tau \leq 2T$ and $\tfrac 12 \leq \sigma \leq \tfrac 12 + (\log T)^{\theta - 3 \e}$,
    	\begin{equation}\label{eq:approx.functional.equation.2}
    		\zeta(\sigma + \ii \tau) = D(\sigma + \ii \tau) + \OO(T^{-1}).
    	\end{equation}
    	Therefore, it suffices to establish \eqref{eqn: int off axis} for $\zeta$ replaced by $D$:
    	\begin{equation}\label{eqn: to prove off axis}
            \int_{-\log^{\theta}T}^{\log^{\theta}T} |D(\sigma + \ii \tau + \ii u)|^{\beta} \rd u \ll \int_{-3\log^{\theta}T}^{3\log^{\theta}T} |D(\tfrac 12 + \ii \tau + \ii u)|^{\beta} \rd u + \frac{1}{(\log T)^{96}}.
    	\end{equation}
	
    	Consider
    	\begin{equation}
    		I(\sigma) =\int_{\mathbb{R}} |D(\sigma + \ii \tau + \ii u)|^\beta \cdot |\Phi_{\Delta, L}(\sigma + \ii u)|^\beta \rd u,
    	\end{equation}
    	with $\Delta = \log^{\e}T$ and $L = 1.5 \log^\theta T$.
        Then, by Lemma~\ref{le:constr}\hspace{0.5mm}(i) and (iv), Corollary~\ref{cor:gabriel} can be applied and yields
    	\begin{equation}\label{eq:prop: LB.eq.1}
            \begin{aligned}
        		&\int_{\mathbb{R}} |D(\sigma + \ii \tau + \ii u)|^{\beta} \cdot |\Phi_{\Delta, L}(\sigma + \ii u)|^{\beta} \rd u \\
                &\hspace{10mm} \ll \int_{\mathbb{R}} |D(\tfrac 12 + \ii \tau + \ii u)|^{\beta} \cdot |\Phi_{\Delta, L}(\tfrac 12 + \ii u)|^{\beta} \rd u.
            \end{aligned}
    	\end{equation}
    	Now, it remains to un-smooth both sides of this expression.
        Lemma~\ref{le:constr}\hspace{0.5mm}(ii) (with $b_1=1/3$) implies that $\Phi_{\Delta, L}(\sigma + \ii u)\gg1$ for $|u|\leq \log^\theta T$. We thus have
    	\begin{equation}\label{eq:prop: LB.eq.2}
            \int_{-\log^\theta T}^{\log^\theta T} |D(\sigma + \ii \tau + \ii u)|^{\beta} \rd u \ll \int_{\mathbb{R}} |D(\sigma + \ii \tau + \ii u)|^{\beta} \cdot |\Phi_{\Delta, L}(\sigma + \ii u)|^{\beta} \rd u,
    	\end{equation}
        settling the left-hand side of \eqref{eqn: to prove off axis}.
        For the right-hand side, note that the choice $\Delta = \log^{\e}T$ and $L = 1.5 \log^\theta T$ ensures that the error term $(\sigma - \tfrac 12) \tfrac{\Delta^2}{L}$ in Lemma~\ref{le:constr} is $(\log T)^{-\e}$ for $\sigma - \tfrac 12 \leq (\log T)^{\theta - 3 \e}$.
    	Lemma~\ref{le:constr}\hspace{0.5mm}(iii) (with $b_2 = 1$) shows that the right-hand side of \eqref{eq:prop: LB.eq.1} is
		\begin{equation}\label{eq:prop: LB.eq.3}
            \begin{aligned}
                &\int_{\mathbb{R}} |D(\tfrac 12 + \ii \tau + \ii u)|^{\beta} \cdot |\Phi_{\Delta, L}(\tfrac 12 + \ii u)|^{\beta} \rd u \\[1mm]
                &\quad\ll \int_{-3\log^\theta T}^{3\log^\theta T} |D (\tfrac 12 + \ii \tau + \ii u)|^{\beta} \rd u + \sum_{\ell=0}^{\infty} \int_{\mathcal{U}_\ell} |D(\tfrac 12 + \ii \tau + \ii u)|^{\beta} \cdot |\Phi_{\Delta, L}(\tfrac 12 + \ii u)|^{\beta} \rd u,
            \end{aligned}
		\end{equation}
        where $\mathcal{U}_\ell =\{3(\log T)^{\theta + \ell} \leq |u| \leq 3 (\log T)^{\theta + \ell + 1}\}$.
        By Corollary~\ref{cor:maxbound} and a union bound, the event
		\begin{equation}
            \mathcal{S}(T) = \bigcap_{\ell=0}^{\infty} \Big\{\max_{|u| \leq \log^{\ell} T} |D(\tfrac 12 + \ii \tau + \ii u)| \leq 2^{\ell} (\log T)^{2 + \ell}\Big\}
		\end{equation}
        has probability $1 - \oo(1)$.
        Moreover, by Lemma~\ref{le:constr} (i) with $A = 1 + \frac{100}{\e} (\lceil \theta \rceil + 1) (1 + 1/\beta)$ and $b_2 = 2 (\log T)^\ell - 1$, we have, for all $3 (\log T)^{\theta + \ell} \leq |u|$,
		\begin{equation}
            |\Phi_{\Delta, L}(\tfrac 12 + \ii u)| \ll (\log T)^{- 100 \ell (1 + 1/\beta)} \cdot (\log T)^{- 100 (\lceil \theta \rceil + 1) \cdot (1 + 1 / \beta)}.
		\end{equation}
		Therefore, on the event $\mathcal{S}(T)$, and for every integer $\ell\geq 0$, the following holds
        \begin{equation}
            \begin{aligned}
                &\int_{\mathcal{U}_\ell} |D(\tfrac 12 + \ii \tau + \ii u )|^{\beta} \cdot |\Phi_{\Delta, L}(\tfrac 12 + \ii u)|^{\beta} \rd u \\[0.5mm]
                &\hspace{10mm}\ll (\log T)^{\lceil \theta \rceil + \ell + 1} \cdot 2^{\beta \ell} (\log T)^{\beta (2 + \ell)} \cdot (\log T)^{-100 (\beta + 1) (\lceil \theta \rceil + \ell + 1)} \\[1mm]
                &\hspace{10mm}\ll (\log T)^{-96 (\beta + 1)\cdot (\lceil \theta \rceil + \ell + 1)}.
            \end{aligned}
		\end{equation}
		Thus, on  $\mathcal{S}(T)$, the contribution of the sum on the right-hand side of \eqref{eq:prop: LB.eq.3} is negligible.
        The claim follows by combining Equations \eqref{eq:prop: LB.eq.1}, \eqref{eq:prop: LB.eq.2} and \eqref{eq:prop: LB.eq.3}.
    \end{proof}

    \subsection{Mollification}\label{sect: mollification}

    This step is an adaptation of Section 4.2 of \cite{ABBRS_2019}, which is itself based on the work of \cite{RadSou15}.
    The treatment is slightly different as the width of the interval needs to be taken into account.
    Also, we choose to use the discretization in Proposition~\ref{prop:union} to obtain a uniform control on the interval as opposed to a Sobolev inequality.

    The main idea is to define a mollifier for the zeta function
    \begin{equation}\label{eqn: M}
        M(s) = \sum_{\substack{\Omega(n) \leq \nu_{\theta} \\ p|n \implies p \leq X}} \frac{\mu(n)}{n^s},
    \end{equation}
    where
    \begin{equation}\label{def:nu}
        \begin{aligned}
            &X = \exp((\log T)^{1-K^{-1}}), ~~ \text{for } K\geq 2, \quad \text{and } \nu_{\theta} = 100 K e^{\theta \vee 0}\log\log T.
        \end{aligned}
    \end{equation}
    Here $\mu$ denotes the M\"obius function $\mu(n)=(-1)^{\omega(n)}$ if $n$ is square-free, where $\omega(n)$ is the number of distinct prime factors, and
    $\mu(n)=0$ if $n$ is non-square-free.
    The estimate will be done slightly off-axis:
    \begin{equation}\label{eqn: sigma_0}
        \sigma_0 = \frac{1}{2} + \frac{(\log T)^{3/(2K)}}{\log T}.
    \end{equation}
    The parameter $K$ will eventually be assumed to be large enough depending on $\theta$, $\beta$ and $\e$.

    The goal of this section is to prove that $M$ is an approximate inverse of $\zeta$:
    \begin{lemma}\label{lem:mollification}
        Let $\theta > -1$ and $\e > 0$ be given. Then,
        \begin{equation}\label{eq:lem:mollification}
            \P\bigg(\max_{|h|\leq \log^\theta T}\left|(\zeta\cdot M)(\sigma_0 + \ii \tau + \ii h) - 1\right| > \e\bigg) = \oo(1).
        \end{equation}
    \end{lemma}

    This was proved in the case $\theta = 0$ in Lemma 4.2 of \cite{ABBRS_2019}.
    In particular, it also holds {\it verbatim} for $-1 < \theta < 0$ since the interval is just smaller.
    The proof of Lemma~\ref{lem:mollification} also holds in the case $\theta > 0$ with slight modifications that we highlight.
    The key idea is the following $L^2$-control:

    \begin{lemma}\label{lem4.2}
        Let $\theta > 0$ be given.
        Then,
        \begin{equation}
            \E\Big[\big| (\zeta\cdot M)(\sigma_0 + \ii \tau) -1\big|^2\Big] \ll (\log T)^{-100 \, e^\theta}.
        \end{equation}
    \end{lemma}

    \begin{proof}
         The proof follows \cite{ABBRS_2019} with a new error term due to the choice of $\nu_{\theta}$.
        (The manipulations are very similar to the ones in Lemma~\ref{le:2nd}.)
        The error appears after Equation (4.10) in \cite{ABBRS_2019} and is given by
        \begin{equation}
            (\log T) \, e^{-\nu_{\theta}}\prod_{p\leq X}(1+7p^{-1}).
        \end{equation}
        The Euler product is $\ll (\log T)^7$ using Lemma~\ref{lem:PNT.estimates}. Using this and the definition of $\nu_{\theta}$ in \eqref{def:nu} yields
        \begin{equation}
            (\log T) \, e^{-\nu_{\theta}}\prod_{p\leq X}(1+7p^{-1})\ll (\log T)^8 \cdot (\log T)^{-100 K e^{\theta}}.
        \end{equation}
        Since $K \geq 2$, this gives the correct estimate.
        Note that the expression $\sum_{p > X}\log (1 - p^{-2\sigma_0})^{-1}$
        entering in the remainder of the proof of Lemma 4.2 in \cite{ABBRS_2019} is
        \begin{equation}
            \ll \sum_{p>X}p^{-2\sigma_0}\ll X^{-(\sigma_0 -1/2)}=\exp(-(\log T)^{\frac{1}{2K}}) \ll (\log T)^{-100 \, e^\theta}.
        \end{equation}
        This ends the proof.
    \end{proof}

    \begin{proof}[Proof of Lemma~\ref{lem:mollification} for $\theta > 0$]
        By Lemma~\ref{zetapprox}, $\zeta$ is well approximated by a Dirichlet polynomial of length $T^{1 + \e}$ for any given $\e > 0$.
        Moreover, $M$ is a Dirichlet polynomial of length less than $T^\e$ for any given $\e > 0$. Therefore, an application of Markov's inequality and Proposition~\ref{prop:union} yield that the probability in \eqref{eq:lem:mollification} is
        \begin{equation}
            \ll \log^{1+\theta} T \cdot  \E\Big[\big| (\zeta\cdot M)(\sigma_0 + \ii \tau) -1\big|^2\Big].
        \end{equation}
        The conclusion follows from Lemma~\ref{lem4.2}.
    \end{proof}

    \subsection{Approximation of the mollifier}\label{sect: bounds D}

    We now approximate the mollifier $M$ by the exponential of a Dirichlet polynomial.
    If we let
    \begin{equation}
        \widetilde{\mathcal{P}}_{1 - K^{-1}}(s) = \sum_{k\geq 1} \sum_{p\leq X} \frac{1}{k p^{k s}},
    \end{equation}
    then the following relation between $\exp(-\widetilde{\mathcal{P}}_{1 - K^{-1}}(s))$ and $M(s)$ holds for all $\Re s \geq 1/2$:
    \begin{equation}
        \exp(-\widetilde{\mathcal{P}}_{1 - K^{-1}}(s)) = \exp\bigg(\log \prod_{p\leq X}(1-p^{-s})\bigg) = M(s) + \sum_{\substack{\Omega(n) > \nu_{\theta} \\ p|n \implies p \leq X}} \frac{\mu(n)}{n^s}.
    \end{equation}
    In particular, we see that $\exp(-\widetilde{\mathcal{P}}_{1 - K^{-1}}(s))$ and $M(s)$ only differ for integers $n$ with more than $\nu_{\theta}$ prime factors ($\Omega(n) > \nu_{\theta}$) and all their prime factors $\leq X$. The following lemma make use of this fact to estimate how close they are when $s = \sigma_0 + \ii \tau + \ii h$.
    \begin{lemma}\label{lem:Lemma.4.5.ABBRS.2018.analog}
        Let $\theta > -1$ be given. Then, for any $K\geq 2$, we have
        \begin{equation}\label{eq:lem:Lemma.4.5.ABBRS.2018.analog}
            \P\bigg(\max_{|h| \leq \log^\theta T} \Big|(M - \exp(-\widetilde{\mathcal{P}}_{1 - K^{-1}}))(\sigma_0 + \ii \tau + \ii h)\Big| > (\log T)^{-10}\bigg) = \oo(1).
        \end{equation}
    \end{lemma}

    \begin{proof}
        The discretization in Proposition~\ref{prop:union} together with the mean value theorem in Lemma~\ref{lem:ABBRS.2018.Lemma.3.3} yield
        \begin{equation}
            \E\bigg[\max_{|h|\leq \log^\theta T} \big|M - \exp(-\widetilde{\mathcal{P}}_{1 - K^{-1}})\big|^2(\sigma_0+\ii \tau+\ii h)\bigg] \ll \log^{1+\theta} T \cdot \sum_{\substack{\Omega(n) > \nu_{\theta} \\ p|n \implies p \leq X}} \hspace{-2mm} n^{-1} .
        \end{equation}
        The right-hand side is $\ll (\log T)^{-100}$ by Rankin's trick and Lemma~\ref{lem:PNT.estimates}:
        \begin{equation}
            \begin{aligned}
                \log^{1+\theta} T \cdot \sum_{\substack{\Omega(n) > \nu_{\theta} \\ p|n \implies p \leq X}} \hspace{-2mm} n^{-1}
                &\ll \log^{1+\theta} T \cdot e^{-\nu_\theta} \hspace{-2mm} \sum_{p|n\implies p\leq X} \hspace{-2mm} e^{\Omega(n)}n^{-1} \\[-4mm]
                &\ll \log^{1+\theta} T \cdot e^{-\nu_\theta} \hspace{-2mm} \prod_{p|n\implies p\leq X} \bigg(1 + \sum_{k\geq 1} \frac{e^k}{p^k}\bigg) \ll (\log T)^{-100}.
            \end{aligned}
        \end{equation}
        The result follows by Markov's inequality.
    \end{proof}

    \newpage
    \subsection{Proofs of the lower bounds}\label{sect: lower proofs}

    Consider, for $0\leq j\leq K-2$, the Dirichlet polynomials
    \begin{equation}\label{eq:def:Dirichlet.polynomials.close.relatives}
       P_j(h) = \Re \sum_{p\in J_j} \frac{1}{p^{\sigma_0 + \ii \tau + \ii h}} , \qquad  J_j = (\exp((\log T)^{\frac{j}{K}}, \exp((\log T)^{\frac{j+1}{K}})].
    \end{equation}
    We choose a probabilistic notation for the increments $P_j$'s seen as random variable, omitting the dependence on the random $\tau$.
    We first prove a lower bound for the moments of Dirichlet polynomials.
    \begin{proposition}\label{prop:lower.bound.Dirichlet.polynomial.moments}
      	Let $\theta > -1$ and $\e > 0$ be given.
        Then,
        \begin{equation}\label{eq:prop:lower.bound.Dirichlet.polynomial.moments}
            \P\Big(\int_{-\log^{\theta} T}^{\log^{\theta} T} \exp\big(\beta \sum_{j=1}^{K-3} P_j(h) \big) \rd h > (\log T)^{f_{\theta}(\beta) - \e}\Big)=1-\oo(1).
        \end{equation}
    \end{proposition}
    The polynomial $P_{K-2}$ is not included in the sum to ensure that the variances of the $P_j$'s are almost equal.
    Indeed, for all $|h| \leq \log^{\theta} T$ and $j\leq K-3$, an application of \eqref{eqn: estimate gaussian moment} yields
    \begin{equation}\label{eqn: s_j}
        s_j^2 =\E[P_j(h)^2] = \frac{1}{2K} \log\log T + \OO((\log T)^{-\frac{1}{2K}}),
    \end{equation}
    since $\sigma_0 - \tfrac{1}{2} = (\log T)^{-1 + 3/(2K)}$. The polynomial $P_{0}$ is ignored to ensure that the polynomials $\sum_{j=1}^{K-3} P_j(h)$ are almost independent for $h$'s that are far apart, which will be crucial for the second-moment method to go through; see below \eqref{eq:second.moment.decomposition} in the proof of Proposition~\ref{prop:lower.bound.Dirichlet.polynomial.moments}.

    \begin{proof}[Proof of Proposition~\ref{prop:lower.bound.Dirichlet.polynomial.moments}]
        This is similar to the upper bound proof of Theorem~\ref{thm: freezing}.
        We first relate the moments to the measure of high points.
        Let $\e > 0$ and $M\in \N$, and set
        \begin{equation}
        	\mathcal{E}_{\theta}(\gamma) \leqdef
            \begin{cases}
        		\theta - \frac{\gamma^2}{1 + \theta}, & \text{ if } \theta \leq 0, \\[1mm]
        		\theta - \gamma^2, & \text{ if } \theta > 0.
        	\end{cases}	
        \end{equation} Consider $\gamma_j = \frac{j}{M} m(\theta) + \e$ for $1 \leq j \leq M$, and the good event
        \begin{equation}
            \begin{aligned}
                E
                &=\bigcap_{j=1}^M \bigg\{ \m\big\{|h| \leq \log^{\theta} T : \exp\big(\sum_{\ell=1}^{K-3} P_{\ell}(h)\big) > (\log T)^{\gamma_{j-1}}\big\} \geq (\log T)^{\mathcal{E}_{\theta}(\gamma_{j-1}) - \e/2}\bigg\} \\[-2mm]
                &\hspace{12mm}\bigcap \bigg\{ \max_{|h| \leq \log^{\theta} T} \, \exp\big(\sum_{\ell=1}^{K-3} P_{\ell}(h)\big) \leq (\log T)^{m(\theta) + \e}\bigg\}.
            \end{aligned}
        \end{equation}
        We will show below that $\P(E)$ is $1-\oo(1)$. Before, we prove the lower bound on the moments on the event $E$.
        We have
        \begin{equation}\label{eq:prop:lower.bound.Dirichlet.polynomial.moments.eq.integration.by.parts}
            \frac{\log \int_{-\log^{\theta} T}^{\log^{\theta} T} \exp\big(\beta \sum_{j=1}^{K-3} P_j(h)\big) \rd h}{\log \log T} \geq \max_{1 \leq j \leq M} \{\beta \gamma_{j-1} + \mathcal{E}_{\theta}(\gamma_{j-1})\} - \e/2.
        \end{equation}
        By the continuity of the function $\gamma \mapsto \beta \gamma + \mathcal{E}_{\theta}(\gamma)$, Equation~\eqref{eq:prop:lower.bound.Dirichlet.polynomial.moments.eq.integration.by.parts} implies that, on the event $E$ and for $M$ large enough with respect to $\e$ and $\beta$,
        \begin{equation}\label{eq:prop:lower.bound.Dirichlet.polynomial.moments.eq.optimization}
            \frac{\log \int_{-\log^{\theta} T}^{\log^{\theta} T} \exp\big(\beta \sum_{j=1}^{K-3} P_j(h)\big) \rd h}{\log \log T} > \max_{\gamma\in [\e,m(\theta)]} \big\{\beta \gamma + \mathcal{E}_{\theta}(\gamma)\big\} - \e.
        \end{equation}
        When $0 < \beta \leq 2 m(\theta) / (1 + (\theta \wedge 0))$, take $\e > 0$ small enough so that $\beta > 2 \e / (1 + (\theta \wedge 0))$.
        The maximum is attained at $\gamma = \tfrac{\beta}{2} (1 + (\theta \wedge 0))$, in which case the right-hand side of \eqref{eq:prop:lower.bound.Dirichlet.polynomial.moments.eq.optimization} is equal to $\tfrac{\beta^2}{4} (1 + (\theta \wedge 0)) + \theta - \e$.
        When $\beta >2 m(\theta) / (1 + (\theta \wedge 0))$, the maximum is attained at $\gamma = m(\theta)$, in which case the right-hand side of \eqref{eq:prop:lower.bound.Dirichlet.polynomial.moments.eq.optimization} is equal to $(\beta m(\theta) - 1) - \e$.
        Thus, on the event $E$ and for $M$ large enough, the lower bound in \eqref{eq:prop:lower.bound.Dirichlet.polynomial.moments} is satisfied.

        \vspace{2mm}
        To conclude the proof of the proposition, it remains to show that $\P(E)\to 1$ as $T\to \infty$.
        By the upper bound on the maximum of $\sum_{j=1}^{K-3} P_j(h)$ in \eqref{eqn: max DP} (and the remark below it for $\theta<0$), it is sufficient to prove that, for all $\eta > 0$ and all $0 < \gamma < m(\theta)$, the event
        \begin{equation}\label{eq:prop:lower.bound.Dirichlet.polynomial.moments.eq.step.1.end}
            \bigg\{\m\Big\{|h| \leq \log^{\theta} T : \sum_{j=1}^{K-3} P_j(h) > \gamma\log\log T\Big\} \geq (\log T)^{\mathcal{E}_{\theta}(\gamma) - \eta}\bigg\}
        \end{equation}
        has probability $1 - \oo(1)$.

        \vspace{2mm}
        Consider
        \begin{equation}
            \mathcal J(\theta) =
            \begin{cases}
                1, &\mbox{if } \theta\geq 0, \\
                \lfloor K |\theta| \rfloor + 1, &\mbox{if } -1<\theta < 0.
            \end{cases}
        \end{equation}
        For $\theta < 0$, Corollary~\ref{cor:DP.cutting.0.to.theta} ensures that the primes up to $\exp(\log^{|\theta|}T)$ only make a very small contribution, namely the event
        \begin{equation}\label{eq:up.to.scale.theta.negligible.contribution}
            \bigg\{\max_{|h|\leq \log ^\theta T}\Big|\sum_{j=1}^{\mathcal{J}(\theta) - 1} P_j(h)\Big| \leq \frac{\gamma}{(1 + \theta) \, K} \log \log T\bigg\}
        \end{equation}
        has probability $1 - \oo(1)$.
        We consider the random variable
        \begin{equation}
            \mathcal{N} = \m\Big\{|h| \leq \log^{\theta} T : P_j(h) > x_j, \text{ for } \mathcal J(\theta)\leq j\leq K-3\Big\},
        \end{equation}
        where
        \begin{equation}\label{eq:x.j.s}
            x_j = \Big(1 + \frac{100}{(1 + (\theta \wedge 0)) \, K}\Big) \cdot \frac{\gamma}{(1 + (\theta \wedge 0)) \, K} \log \log T.
        \end{equation}
        By summing the $x_j$'s, it is not hard to check that the intersection of the events $\{\mathcal{N}\geq (\log T)^{\mathcal{E}_{\theta}(\gamma) - \eta}\}$  and the one in \eqref{eq:up.to.scale.theta.negligible.contribution} is included in the event in \eqref{eq:prop:lower.bound.Dirichlet.polynomial.moments.eq.step.1.end}.
        Therefore, the proof of the proposition is reduced to show
        \begin{equation}\label{eq:to.prove.with.Paley.Zygmund}
            \P\big(\mathcal{N}\geq (\log T)^{\mathcal{E}_{\theta}(\gamma) - \eta}\big) = 1 - \oo(1) .
        \end{equation}
        This is established by the Paley-Zygmund inequality.

        \vspace{4mm}
        To this aim, we shall need one-point and two-point large deviation estimates for the event
        \begin{equation}\label{eqn: A}
            A (h) = \Big\{P_j( h) > x_j, ~\text{for } \mathcal J(\theta) \leq j \leq K-3\Big\},\  \theta>-1,\ h,h'\in[- \log^\theta T, \log^\theta T].
        \end{equation}
        The next two propositions are stated as Propositions 5.4 and 5.5 in \cite{ABBRS_2019}.
        They are consequences of the Gaussian moments in Lemma~\ref{lem: moments}.

        \begin{proposition}[One-point large deviation estimates]\label{prop: two point couple}
            Consider the event $A(h)$ in \eqref{eqn: A}.
            For any choices of $\sqrt{\log \log T} \ll_K x_j \leq \log\log T$ where $1 \leq j\leq K-3$, and uniformly for $h,h'\in [-\log^{\theta} T, \log^{\theta} T]$, we have
            \begin{equation}\label{eqn: prob one point}
                \P(A(h)) = (1 + \oo(1)) \prod_{j=\mathcal{J}(\theta)}^{K-3} \int_{x_j/s_j}^{\infty} \frac{e^{-y^2/2}}{\sqrt{2\pi}} \rd y
                \asymp \prod_{j = \mathcal J(\theta)}^{K-3} \frac{s_j}{x_j}\cdot e^{ -x_j^2/(2s_j^2)}.
            \end{equation}
        \end{proposition}

        In the case of two points $h,h'$, the primes are essentially correlated up to $\exp(|h-h'|^{-1})$ and quickly decorrelate afterwards.
        For $\theta\geq 0$, this means that the $P_j$'s are essentially independent whenever $|h-h'| > (\log T)^{-\frac{1}{2K}}$, since $j=0$ is excluded.
        For $\theta < 0$, we must exclude the $j$'s up to $\mathcal J(\theta) - 1$. Therefore, the $P_j$'s are essentially independent whenever $|h-h'| > (\log T)^{\theta - \frac{1}{2K}}$.
        We get:

        \begin{proposition}[Two-point large deviation estimates]\label{prop: prob decouple}
            Consider the event $A(h)$ in \eqref{eqn: A}.
            For any choices of $0 < x_j \leq \log\log T$, and uniformly for $h,h'\in [-\log^{\theta} T, \log^{\theta} T]$ such that $|h-h'| > (\log T)^{\scriptscriptstyle -\frac{\mathcal{J}(\theta)}{K} + \frac{1}{2K}}$, we have
            \begin{equation}\label{eqn: prob two points}
                \P(A(h) \cap A (h')) = (1 + \oo(1)) \, \P(A(h))\, \P(A(h')).
            \end{equation}
            Furthermore, let $0 \leq \ell \leq K-3$.
            Then, uniformly for $h,h'\in [-\log^{\theta} T, \log^{\theta} T]$ such that $|h-h'| \leq (\log T)^{-\ell/K}$, we have
            \begin{equation}\label{eqn: two point mesoscopic}
                \P(A(h) \cap A (h'))
                \ll \exp\bigg(-\sum_{j = \mathcal J(\theta)}^\ell \frac{x_j^2}{2s_j^2} ~- \sum_{j = (\ell+1)\vee \mathcal {J(\theta)}}^{K-3} \frac{x_j^2}{s_j^2}\bigg).
            \end{equation}
        \end{proposition}

        Now, in order to prove \eqref{eq:to.prove.with.Paley.Zygmund}, we start by finding a lower bound on $\E[\mathcal{N}]$.
        By \eqref{eqn: prob one point}, the $x_j$'s in \eqref{eq:x.j.s} and the $s_j$'s in \eqref{eqn: s_j}, we have
        \begin{equation}\label{eqn: lb one point}
            \E[\mathcal{N}] = \int_{-\log^{\theta} T}^{\log^{\theta} T} \P(A(h)) \rd h \gg \log^{\theta} T \prod_{j = \mathcal J(\theta)}^{K-3} \frac{s_j}{x_j}\cdot e^{ -x_j^2/(2s_j^2)}
            \gg (\log T)^{\mathcal{E}_{\theta}(\gamma) - \eta/3},
        \end{equation}
        assuming that $K$ is large enough with respect to $\theta$, $\gamma$ and $\eta$.
        By the Paley-Zygmund inequality, this implies
        \begin{equation}
            \begin{aligned}
                \P\big(\mathcal{N}\geq (\log T)^{\mathcal{E}_{\theta}(\gamma) - \eta}\big)
                &\geq \P\big(\mathcal{N}\geq (\log T)^{-\eta/3}\E[\mathcal{N}]\big) \\
                &\geq \big(1 - (\log T)^{-\eta/3}\big) \frac{(\E[\mathcal{N}])^2}{\E[\mathcal{N}^2]}.
            \end{aligned}
        \end{equation}
        It remains to show $\E[\mathcal{N}^2] = (1 + \oo(1))(\E[\mathcal{N}])^2$.
        With $I = [-\log^{\theta} T, \log^{\theta} T]$, Fubini's theorem yields
        \begin{equation}\label{eq:second.moment.decomposition}
            \E[\mathcal{N}^2] = \int_{I \times I} \P(A(h) \cap  A(h')) \, \rd h \rd h'.
        \end{equation}
        The integral can be divided into $(K - \mathcal{J}(\theta) + 1)$ parts:
        \begin{equation}
            \begin{aligned}
                &B = \{(h,h'): |h-h'| > (\log T)^{-\frac{\mathcal{J}(\theta)}{K} + \frac{1}{2K}}\}; \\[-0.5mm]
                &B_0 = \{(h,h'): (\log T)^{-\frac{\mathcal{J}(\theta)}{K}} < |h-h'| \leq (\log T)^{-\frac{\mathcal{J}(\theta)}{K} + \frac{1}{2K}}\}; \\[0.5mm]
                &B_\ell = \{(h,h'): (\log T)^{-(\ell+1)/K} < |h-h'| \leq (\log T)^{-\ell/K}\}, \quad \text{for } \ell =\mathcal J(\theta),\dots, K-3; \\[0mm]
                &B_{K-2} = \{(h,h'): |h-h'| \leq (\log T)^{-(K-2)/K}\}.
            \end{aligned}
        \end{equation}
        The dominant term will be the one on $B$.
        Note that $\m(B) = \m(I)^2 (1 + \oo(1))$. Hence, by \eqref{eqn: prob two points}, we have
        \begin{equation}
            \int_B \P(A(h)\cap A(h')) \, \rd h\rd h' = (1 + \oo(1)) (\E[\mathcal{N}])^2.
        \end{equation}
        By \eqref{eqn: two point mesoscopic} and the estimate \eqref{eqn: lb one point}, the integral on $B_0$ is
        \begin{equation}
            \begin{aligned}
                &\ll (\log T)^{\theta - \frac{\mathcal{J}(\theta)}{K} + \frac{1}{2K}} \exp\bigg(\sum_{j =\mathcal J(\theta)}^{K-3} - \frac{x_j^2}{s_j^2}\bigg) \ll (\log T)^{-(\theta \vee 0) - \frac{1}{3K}} (\E[\mathcal{N}])^2,
            \end{aligned}
        \end{equation}
        assuming that $K$ is large enough with respect to $\theta$ and $\gamma$.
        For $\ell =\mathcal J(\theta), \dots, K-3$, the integral on $B_\ell$ is, by \eqref{eqn: two point mesoscopic} and the estimate \eqref{eqn: lb one point},
        \begin{equation}\label{eq:prop:lower.bound.Dirichlet.polynomial.moments.eq.step.3.decomposition.m.K}
            \begin{aligned}
                &\ll (\log T)^{\theta-\ell/K} \exp\bigg(-\sum_{j = \mathcal J(\theta)}^\ell \frac{x_j^2}{2s_j^2} \, - \sum_{j = \ell + 1}^{K-3} \frac{x_j^2}{s_j^2}\bigg) \\
                &= (\log T)^{-\theta - \ell/K} \exp\bigg(\sum_{j = \mathcal J(\theta)}^{\ell}\frac{x_j^2}{2s_j^2}  \bigg) \cdot (\log T)^{2\theta} \exp\bigg(-\sum_{j = \mathcal J(\theta)}^{K-3} \frac{x_j^2}{s_j^2} \bigg) \\
                &\ll (\log T)^{-\theta - \ell/K + (\ell/K + (\theta \wedge 0)) \frac{\gamma^2}{(1 + (\theta \wedge 0))^2} + \eta} \, (\E[\mathcal{N}])^2,
            \end{aligned}
        \end{equation}
        assuming again that $K$ is large enough with respect to $\theta$, $\gamma$ and $\eta$.
        Since $\gamma^2 < m(\theta)^2 = (1 + \theta)(1 + (\theta \wedge 0))$, the right-hand side of \eqref{eq:prop:lower.bound.Dirichlet.polynomial.moments.eq.step.3.decomposition.m.K} is $\oo\big((\E[\mathcal{N}])^2\big)$ if we fix $\eta > 0$ small enough with respect to $\theta$ and $\gamma$.
        Similarly, by \eqref{eqn: prob one point} and the estimate \eqref{eqn: lb one point}, the integral on $B_{K-2}$ is
        \begin{equation}
            \leq  \int_{B_{K-2}} \ \P(A(h)) \, \rd h\rd h' \ll (\log T)^{- 1 + 2/K + \eta/3}\cdot \E[\mathcal{N}] = \oo\big((\E[\mathcal{N}])^2\big),
        \end{equation}
        provided that $\eta$ is small enough with respect to $\theta$ and $\gamma$, and $K$ is large enough with respect to $\theta$, $\gamma$ and $\eta$.
        This concludes the proof of Proposition~\ref{prop:lower.bound.Dirichlet.polynomial.moments}.
    \end{proof}

    Putting all the work of Section 3 together, we can prove the lower bound in Theorem~\ref{thm: freezing}.

    \begin{proof}[Proof of Proposition~\ref{prop:lower.bound.moments.zeta}]
        By Proposition~\ref{prop: LB}, the probability in \eqref{eq:prop:lower.bound.moments.zeta} is
        \begin{equation}\label{eq:prop:lower.bound.moments.zeta.eq.1}
            \geq \P\Big(\int_{- \frac 13 \log^{\theta} T}^{\frac 13 \log^{\theta} T} |\zeta(\sigma_0 + \ii \tau + \ii h)|^{\beta} \rd h > (\log T)^{f_{\theta}(\beta) - \e}\Big) - \oo(1).
        \end{equation}
        By Lemma~\ref{lem:mollification} and Lemma~\ref{lem:Lemma.4.5.ABBRS.2018.analog}, the above is
        \begin{equation}\label{eq:prop:lower.bound.moments.zeta.eq.2}
            \geq \P\Big(\int_{- \frac 13 \log^{\theta} T}^{\frac 13 \log^{\theta} T} \exp\big(\beta \, \Re \widetilde{\mathcal{P}}_{1-K^{-1}}(\sigma_0 + \ii \tau + \ii h)\big) \rd h > (\log T)^{f_{\theta}(\beta) - 2\e}\Big) - \oo(1).
        \end{equation}
        Now, notice that the (double) sum for $k\geq 3$ in $\widetilde{\mathcal{P}}_{1-K^{-1}}(\sigma_0 + \ii \tau + \ii h)$ is of order one (uniformly for $|h| \leq \log^{\theta} T$), and that the sum for $k=2$ is of negligible order:
        \begin{equation}\label{eqn: power neglect}
            \P\bigg(\max_{|h| \leq \log^\theta T} \Big|\sum_{p\leq X} \tfrac{1}{2} \, p^{-2(\sigma_0+\ii\tau+\ii h)}\Big| > A\bigg)
            \ll A^{-2\ell}(\log^{1+\theta} T )\cdot \ell \hspace{0.3mm}! \, ,
        \end{equation}
        where we use the discretization from Proposition~\ref{prop:union} and the moment estimates from Lemma~\ref{lem:moment.estimates.Lemma.3.Sound.2009.generalisation}.
        Indeed, the right-hand side of \eqref{eqn: power neglect} is $\oo(1)$ with the choice $A = \sqrt{\nu_{\theta}}$ and $\ell = \lfloor (1+\theta)\log\log T \rfloor$.
        Hence, $\widetilde{\mathcal{P}}_{1-K^{-1}}$ can be replaced by $\mathcal{P}_{1-K^{-1}}$ with an error less than $\log^\e T$ with probability $1 - \oo(1)$, meaning that the right-hand side of \eqref{eq:prop:lower.bound.moments.zeta.eq.2} is
        \begin{equation}\label{eq:prop:lower.bound.moments.zeta.eq.3}
            \geq \P\Big(\int_{- \frac 13 \log^{\theta} T}^{\frac 13 \log^{\theta} T} \exp\big(\beta \, \Re \mathcal{P}_{1-K^{-1}}(\sigma_0 + \ii \tau + \ii h)\big) \rd h > (\log T)^{f_{\theta}(\beta) - 3\e}\Big) - \oo(1).
        \end{equation}
        By \eqref{eqn: max DP}, we may discard the terms with $j = 0$ and $j = K-2$ with a similar error.
        For $K$ large enough with respect to $\e$, $\beta$ and $\theta$, the probability in \eqref{eq:prop:lower.bound.moments.zeta.eq.3} is therefore
        \begin{equation}\label{eq:prop:lower.bound.moments.zeta.eq.4}
            \geq \P\Big(\int_{ - \frac 13 \log^{\theta} T}^{\frac 13 \log^{\theta} T} \exp\big(\beta \sum_{j=1}^{K-3} P_j(h)\big) \rd h > (\log T)^{f_{\theta}(\beta) - 4\e}\Big) - \oo(1).
        \end{equation}
        Finally, the probability in \eqref{eq:prop:lower.bound.moments.zeta.eq.4} tends to $1$ as $T\to\infty$ by Proposition~\ref{prop:lower.bound.Dirichlet.polynomial.moments}.
    \end{proof}

    We now prove the lower bound in Theorem~\ref{thm: max}.

    \begin{proof}[Proof of Proposition~\ref{prop:lower.bound.maximum.zeta}]
        From \eqref{eq:def:free.energy}, we have that $f_{\theta}(\beta) = \beta m(\theta)- 1$ when $\beta > \beta_c(\theta) = 2\sqrt{1 + (\theta \wedge 0)}$.
        Thus, on the event in the statement of Proposition~\ref{prop:lower.bound.moments.zeta} (which has probability $1 - \oo(1)$), and for $\beta$ large enough with respect to $\e$ and $\theta$, we have
        \vspace{-2mm}
        \begin{equation}\label{eq:prop:lower.bound.zeta.maximum.eq.1}
            \begin{aligned}
                \max_{|h| \leq \log^{\theta} T} |\zeta(\tfrac 12 + \ii \tau + \ii h)|
                &\geq \left(\frac{1}{2 \log^{\theta} T} \int_{-\log^{\theta} T}^{\log^{\theta} T} |\zeta(\tfrac 12 + \ii \tau + \ii h)|^{\beta} \rd h\right)^{\hspace{-0.5mm}1/\beta} \\[1.5mm]
                &\gg (\log T)^{m(\theta)- \frac{(1 + \e + \theta)}{\beta}} \\[1mm]
                &\geq (\log T)^{m(\theta)- \e}.
            \end{aligned}
        \end{equation}
        This ends the proof.
    \end{proof}

\appendix

\section{Useful estimates}

The prime number theorem yields estimates on the sum of primes with a good error.

\begin{lemma}\label{lem:PNT.estimates}
    Let $1 \leq P \leq Q$, then
    \begin{equation}\label{eq:lem:PNT.estimates.1}
        \sum_{P < p \leq Q} \frac{(\log p)^m}{p} =
        \begin{cases}
            \frac{(\log Q)^m}{m} - \frac{(\log P)^m}{m} + \OO_m(1), &\mbox{if } m \geq 1, \\[1.5mm]
            \log \log Q - \log \log P + \OO(e^{-c\sqrt{\log P}}), &\mbox{if } m = 0.
        \end{cases}
    \end{equation}
    Also, for $|\eta \log Q| \leq 1$,
    \begin{equation}\label{eq:lem:PNT.estimates.3}
        \sum_{P < p \leq Q} \frac{\cos(\eta \log p)}{p} = \log \log Q - \log \log P + \OO(1).
    \end{equation}
\end{lemma}

\begin{proof}
    For \eqref{eq:lem:PNT.estimates.1}, see Lemma A.1 in \cite{MR3906393} and Lemma 2.1 in \cite{MR3619786}.
    For \eqref{eq:lem:PNT.estimates.3}, see p.20 in \cite{harper_note_13}.
\end{proof}

The next three results yield moment estimates for Dirichlet polynomials.
The first one is an elementary bound. The second ensures that moments of Dirichlet polynomials that are not too high are approximately Gaussian.

\begin{lemma}[Lemma 3.3 in \cite{ABBRS_2019}]\label{lem:ABBRS.2018.Lemma.3.3}
    For any complex numbers $a(n)$ and $b(n)$, and for $N \leq T$, we have
    \begin{equation}
        \begin{aligned}
            &\E\bigg[\Big(\sum_{m\leq N} a(m) m^{-\ii \tau}\Big) \Big(\sum_{n\leq N} b(n) n^{\ii \tau}\Big)\bigg] \\
            &\quad= \sum_{n\leq N} a(n)b(n) + \OO\bigg(\frac{N \log N}{T} \sum_{n\leq N} (|a(n)|^2 + |b(n)|^2)\bigg).
        \end{aligned}
    \end{equation}
\end{lemma}

\begin{lemma}[Lemma 3.4 in  \cite{ABBRS_2019}]\label{lem: moments}
    Let $x\ge 2$ be a real number, and suppose that for primes $p\le x$, $a(p)$ is a complex number with $|a(p)|\leq 1$.
    Then, for any $k\in \N$,
    \begin{equation}
        \begin{aligned}
            &\E\bigg[\Big(\frac{1}{2}\sum_{ p \le x} (a(p)p^{-\ii \tau}+ \overline{a(p)}p^{\ii \tau})\Big)^{k}\bigg] = \frac{\partial^k}{\partial z^k}\Big(\prod_{p\le x} I_0(|a(p)|z)\Big)\Big|_{z=0} + \OO\Big(\frac{x^{2k}}{T}\Big),
        \end{aligned}
    \end{equation}
    where $I_0(z) = \sum_{n\geq 0}z^{2n}/(2^{2n}(n!)^2)$ denotes the modified Bessel function of the first kind of order $0$.
    In particular, the expression is $\OO\left(x^{2k}/T\right)$ for odd $k$.
\end{lemma}

The relation with Gaussian moments in the case where $a(p) = p^{-\sigma - \ii h}$ is obtained by expanding the product to get
\begin{equation}
    \prod_{p\le x} I_0(|a(p)|z) = F(z)\cdot\exp\bigg(\frac{z^2}{2} \cdot \frac{1}{2} \sum_{p\leq x} p^{-2\sigma}\bigg)
\end{equation}
where $F(z)$ is analytic in a neighborhood of $0$ with $F(0)=1$ and any derivative of a fixed order is bounded by $\sum_{p\leq x}p^{-4\sigma}$ uniformly in $z$.
In particular, this implies that, for $\sigma\geq 1/2$ and $k$ small enough so that $x^{2k}/T=\oo(1)$,
\begin{equation}\label{eqn: estimate gaussian moment}
    \E\bigg[\Big(\sum_{p \le x} \Re p^{-\sigma - \ii \tau - \ii h}\Big)^{2k}\bigg]
    = (1+\oo(1)) \, \frac{(2k)!}{2^k \cdot k!} \, \bigg(\frac{1}{2} \sum_{p\leq x}p^{-2\sigma}\bigg)^k.
\end{equation}
The above also holds if $a(p)=0$ for $p\leq y$ (say) with the sum over primes restricted to $y< p\leq x$.
In particular, the error $\sum_{y<p\leq x}p^{-4\sigma}$ can be made $\oo(1)$ by taking $y$ large.
We note that the moments yield a Gaussian tail
\begin{equation}\label{eqn: gaussian tail}
    \P\Big(\sum_{p \le x} \Re p^{-\sigma - \ii \tau - \ii h} > V\Big)\ll \exp(-V^2 /(2 v^2)),
\end{equation}
by picking the moment $k=\lfloor V^2 / 2 v^2 \rfloor$ with $v^2=\frac{1}{2} \sum_{p\leq x}p^{-2\sigma}$, for $V$ not too large.

\vspace{3mm}
Finally the third estimate is a cruder version of the Gaussian moment estimates that yields quick upper bounds on moments.

\begin{lemma}[Lemma 3 in \cite{Sound09}]\label{lem:moment.estimates.Lemma.3.Sound.2009.generalisation}
    Let $T$ be large, and let $2 \leq x \leq T$.
    Let $\ell$ be a natural number such that $x^{\ell} \ll T / \log T$.
    For any complex numbers $a(p)$, we have
    \vspace{-2mm}
    \begin{equation}\label{eq:lem:moment.estimates.Lemma.3.Sound.2009.generalisation}
        \E\bigg[\Big|\sum_{p \leq x} \frac{a(p)}{p^{1/2 + \ii \tau}}\Big|^{2\ell}\bigg] \ll \ell! \, \bigg(\sum_{p \leq x} \frac{|a(p)|^2}{p}\bigg)^\ell.
    \end{equation}
\end{lemma}



%
%

\bibliographystyle{imsart-nameyear}
\bibliography{Arguin_Ouimet_Radziwill_2021_freezing_bib}

\end{document}